\documentclass[reqno,10pt]{amsart}
\title{Weighted words at degree two, I: Bressoud's algorithm as an energy transfer}
\author{Isaac KONAN} \address{IRIF \\ University
of Paris\\  Paris,  75013, France}
\email{konan@irif.fr}
\date{}

\usepackage{amssymb}
\usepackage{amsmath}
\usepackage{amsthm}
\usepackage{todonotes}
\usepackage{fullpage}
\usepackage{enumitem}
\usepackage{xcolor}
\usepackage{graphicx}
\usepackage{pgfarrows,pgfnodes}
\usepackage{url}
\usepackage{textcomp}
\usepackage{tikz}
\usepackage[toc,page]{appendix}

\allowdisplaybreaks
\newcommand{\m}{\medbreak}
\newcommand{\bi}{\bigbreak}

\definecolor{foge}{rgb}{0.1, 0.6, 0.1}

\newcommand{\So}{\textbf{Step 1 }}

\newcommand{\Soo}{\textbf{Step 1}}
\newcommand{\St}{\textbf{Step 2 }}
\newcommand{\Stt}{\textbf{Step 2}}

\numberwithin{equation}{section}
\numberwithin{figure}{section}

\newcommand{\Thm}[1]{Theorem \ref{#1}}
\newcommand{\Cor}[1]{Corollary \ref{#1}}
\newcommand{\Lem}[1]{Lemma \ref{#1}}
\newcommand{\Prp}[1]{Proposition \ref{#1}}
\newcommand{\Sct}[1]{Section \ref{#1}}
\newcommand{\Def}[1]{Definition \ref{#1}}
\newcommand{\Rem}[1]{Remark \ref{#1}}
\newcommand{\Expl}[1]{Example \ref{#1}}
\newcommand{\Expls}[1]{Examples \ref{#1}}

\newcommand{\Ll}{\Lambda}
\newcommand{\Pp}{\mathcal{P}}

\newcommand{\Z}{\mathbb{Z}}

\newcommand{\C}{\mathcal{C}}

\newcommand{\ep}{\epsilon}

\newcommand{\Odd}{\mathcal{O}_{\ep}}
\newcommand{\Sc}{\mathcal{S}_{\ep}}

\newcommand{\od}{\succ_{\ep}}
\newcommand{\odp}{\gg_{\ep}}
\newcommand{\odg}{\gg^{\ep}}
\newcommand \cc {\overline{c}}

\newcommand{\Cr}{\textit{Chasles' relation }}
\newcommand{\Crr}{\textit{Chasles' relation}}
\newcommand{\Ti}{\textit{triangle inequality}}

\newcommand{\Par}{\textit{energetic particle }}

\newcommand{\pt}{\phi}
\newcommand{\E}{\mathcal{E}_{\ep}}
\newcommand{\Ee}{\mathcal{E}^{\ep}}
\newcommand{\la}{\lambda}
\newcommand{\sss}{\{1,\ldots,s\}}
\newcommand{\ssss}{\{1,\ldots,s-1\}}
\newcommand{\Ssss}{\{2,\ldots,s\}}

\numberwithin{equation}{section}
\newtheorem{theo}{Theorem}[section]
\newtheorem{prop}[theo]{Proposition}
\newtheorem{lem}[theo]{Lemma}
\newtheorem{cor}[theo]{Corollary}
\newtheorem{rem}[theo]{Remark}
\newtheorem{ex}[theo]{Example}
\newtheorem{exs}[theo]{Examples}
\theoremstyle{definition} \newtheorem{deff}[theo]{Definition}

\begin{document}
\maketitle
\begin{abstract}
In a recent paper, we generalized a partition identity stated by Siladi\'c in his study of the level one standard module of type $A_2^{(2)}$.  The proof used weighted words with an arbitrary number of primary colors and all the secondary colors obtained from these primary colors, and a brand new variant of the bijection of Bressoud for Schur's partition identity.
In this paper, the first of two, we analyze this variant of Bressoud's algorithm in the framework of statistical mechanics, where an integer partition is viewed as an amount of energy shared, according to certain properties, between several states. This viewpoint allows us to generalize the previous result by considering a more general family of minimal difference conditions. For example, we generalize the Siladi\'c identity to overpartitions.
In the second paper, we connect this result to the Glaisher theorem and give some applications to level one perfect crystals.
\end{abstract}
\section{Introduction}
\subsection{History}
\subsubsection{Weighted words: from Alladi-Gordon to Siladi\'c}
Let $n$ be a positive integer. A partition of $n$ is defined as a non-increasing sequence of positive integers, called the parts of the partition, and whose sum is equal to $n$. For example, the partitions of $5$ are 
\[(5),(4,1),(3,2),(3,1,1),(2,2,1),(2,1,1,1), \,\,\text{and}\,\,(1,1,1,1,1,1,1)\,\cdot\]
By a partition identity we mean a combinatorial identity that links two or several sets of integer partitions. The study of such identities has interested mathematicians for centuries, dating back to
Euler's proof that there are as many partitions of $n$ into distinct parts as partitions of $n$ into odd parts. The Euler distinct-odd identity can be written in terms of $q$-series with the following expression:
\begin{equation}\label{eq:euler}
(-q;q)_\infty = \frac{1}{(q;q^2)_\infty}\,\cdot
\end{equation}
In the latter formula, $(x;q)_m = \prod_{k=0}^{m-1} (1-xq^k)$
for  any $m\in \mathbb{N}\cup \{\infty\}$ and $x,q$ such that $|q|<1$.
\m A broad generalization of Euler's identity was found and proved by Glaisher. In \cite{G83}, Glaisher stated that, for any positive integers $m$ and $n$, there are as many partitions of $n$ into parts not divisible by $m$ as partitions of $n$ with fewer than $m$ occurrences for each positive integer. One can convey the Glaisher identity as the following $q$-series
\begin{equation}
\prod_{n\geq 1} (1+q^n+ q^{2n}+\cdots+q^{n(m-1)}) = \prod_{\substack{n\geq 1 \\ m \nmid n}} \frac{1}{(1-q^n)} = \frac{(q^m;q^m)_{\infty}}{(q;q)_{\infty}} \,\cdot
\end{equation}
\bi The theory of integer partitions underwent significant advancement in the earlier twentieth century. Major works on partitions identities were led by MacMahon \cite{Mac15}, Rogers and Ramanujan \cite{RR19}, and Schur \cite{Sc26}. Schur stated in his work one of the most important identities in the theory of partitions.
\begin{theo}[Schur]\label{theo:schur}
For any positive integer $n$, the number of partitions of $n$ into distinct parts congruent to $\pm 1 \mod 3$ is equal to the number of partitions of $n$ where parts differ by at least three and multiples of three differ by at least six.
\end{theo}
There have been a number of proofs of Schur's result over the years,
including a $q$-difference equation proof of Andrews \cite{AN68} and a simple
bijective proof of Bressoud \cite{BR80}.
\bi 
In the 90's, seminal work of Alladi and Gordon showed how the
theorem of Schur emerges from more general results \cite{AAG95}. They introduced weighted words, a method which consists in associating some colors to the integers, and then considering integer partitions into colored integers. Such partitions are called \textit{colored partitions}. 
\m We consider that the integers occur in three colors $\{a,b,ab\}$, and we order them as follows:
\begin{equation}
1_{ab}<1_a<1_b<2_{ab}<2_a<2_b<3_{ab}<\cdots\,\cdot
\end{equation}
We then consider the partitions with colored parts different from $1_{ab}$ and satisfying the minimal difference conditions in the matrix 
\begin{equation}\label{eq:diffag}
\bordermatrix{
\text{}&ab&a&b
\cr ab&2&2&2
\cr a&1&1&2
\cr b&1&1&1
}\,\cdot
\end{equation}
Here, the term ``minimal difference conditions" means that, for a colored partition $\la = (\la_1,\cdots,\la_s)$, the part $\la_i$ with color in the row and the part $\la_{i+1}$ with color in the column differ by at least the corresponding entry in the matrix.  
An example of such a partition is $(7_{ab},5_{b},4_{a},3_{ab},1_b)$. The Alladi-Gordon refinement of Schur's partition theorem \cite{AG93} is stated as follows:
\begin{theo}[Alladi-Gordon]\label{theo:ag}
Let $u,v,n$ be non-negative integers. Denote by $A(u,v,n)$ the number of partitions of $n$ into $u$ distinct parts with color $a$ and $v$ distinct parts with color $b$, and denote by $B(u,v,n)$ the number of partitions of $n$ satisfying the conditions above, with $u$ parts with color $a$ or $ab$, and $v$ parts with color $b$ or $ab$. We then have  $A(u,v,n)=B(u,v,n)$ and the identity
\begin{equation}
\sum_{u,v,n\geq 0} B(u,v,n)a^ub^vq^n = \sum_{u,v,n\geq 0} A(u,v,n)a^ub^vq^n = (-aq;q)_\infty (-bq;q)_\infty\,\cdot
\end{equation}
\end{theo}
We obtain the Schur theorem by applying the transformation $(q,a,b)\mapsto (q^3,q^{-2},q^{-1})$ in the latter identity.
In fact, the minimal difference conditions given in \eqref{eq:diffag} give after these transformations the minimal differences in Schur's theorem. 
\m The weighted words method appears as a major tool in the study of partition identities. On one hand, it allows us to have a better understanding of the partitions' structure, and gives a hint to find some suitable bijective proofs for the identities. On the other hand, one can generate an unlimited number of new identities by applying transformations on the colors. Subsequent works using this method led to the discovery of several new identities \cite{AAB03,AAG95,CL06,Dousse16,DK1,IK19,IK2}. 
\bi
Another rich source of partition identities is the representation theory of Lie algebras. This was initiated by the work of Lepowsky and Wilson \cite{LW84}, who proved the Rogers-Ramanujan identities
by using representations of level $3$ standard modules of the affine Lie algebra $A_1^{(1)}$. 
Subsequently, Capparelli \cite{C93}, Meurman-Primc \cite{MP87} and others examined
related standard modules and affine Lie algebras and found many new partition identities.
\m 
In \cite{Si02},  Siladi\'c gave the following partition identity in his study of representations of the twisted affine Lie algebra $A_2^{(2)}$.
\begin{theo}[Siladi\'c]\label{theo:siladic}
The number of partitions $\lambda_1+\cdots+\lambda_s$ of an integer $n$ into distinct odd parts 
is equal to the number of partitions of $n$, into parts different from $2$, such that 
$\lambda_i -\lambda_{i+1}\geq 5$ and 
\begin{align*}
\lambda_i -\lambda_{i+1}  = 5 &\Rightarrow \, \lambda_i+\lambda_{i+1} \equiv \pm 3 \mod 16\,,\\
\lambda_i -\lambda_{i+1}  = 6 &\Rightarrow \, \lambda_i+\lambda_{i+1}  \equiv 0,\pm 4, 8\mod 16\,,\\
\lambda_i -\lambda_{i+1}  = 7 &\Rightarrow \, \lambda_i+\lambda_{i+1}  \equiv \pm 1, \pm 5, \pm 7 \mod 16\,,\\
\lambda_i -\lambda_{i+1}  = 8 &\Rightarrow \, \lambda_i +\lambda_{i+1} \equiv 0,\pm 2,\pm 6, 8\mod 16\,\,\cdot
\end{align*}
\end{theo}
This theorem has been refined by Dousse in \cite{Dousse16} where she used weighted words with two primary colors $a,b$ and three secondary colors $a^2,ab,b^2$. Starting from her refinement, the author was able to give in \cite{IK19} a generalization of Siladi\'c's theorem for an arbitrary number $n$ of primary colors $a_1,\ldots,a_n$ along with the set $\{a_ia_j: i,j =1,\ldots,n\}$ of all the $n^2$ secondary colors. He bijectively proved his identity, by using a brand new variant of the algorithm given by Bressoud in his bijective proof of Schur's identity \cite{BR80}.
\m In this paper, we aim at generalizing the result given \cite{IK19}, by using the statistic-mechanical viewpoint of  the integer partitions. 
\subsubsection{Integer partitions in statistical mechanics}
The connection between integer partitions and physics was first pointed out by Bohr and Kalckar \cite{BK37}.
In the same year, Van Lier and
Uhlenbeck noted the links between the problem of counting microstates of the systems obeying Bose or Fermi statistics and some problems related to integer partitions \cite{VU37}. 
\m 
Since then, a current approach in statistical mechanics consists in considering a partition of a given integer into parts with certain restrictions as a sharing of a fixed amount of energy among the different 
possible states of an assembly. This approach can be found in the seminal works of Auluck and Kothari \cite{AK46}, Temperley \cite{TP49} and Nanda \cite{Nan51}.
\bi 
In this paper, we view the weighted words in the framework of statistical mechanics.
We then refer to the colors as \textit{states}, and the sizes of parts as \textit{potentials}. To place the study of weighted words in a more general context, we first need to relax our conditions in the definition of integer partitions. 
\m Let $\C$ be a set of colors, and let $\Z_{\C} = \{k_c : k \in \Z, c \in \C\}$ be the set of colored integers. We recall that we identify the colors as states, and we now refer to the colored integers as energetic particles having a state and a potential.
\begin{deff}
Let $\gg$ be a binary relation defined on $\Z_{\C}$.
A \textit{generalized colored partition} with relation $\gg$ is a finite sequence $(\pi_1,\ldots,\pi_s)$ of energetic particles, where for all $i \in \{1,\ldots,s-1\},$ $\pi_i\gg\pi_{i+1}.$
\end{deff}
In the following, we denote by $c(\pi_i) \in \C$ the state of the particle $\pi_i$. The quantity $|\pi|=\pi_1+\cdots+\pi_s$ is the total size or \textit{Energy} of $\pi$, and $C(\pi) = c(\pi_1)\cdots c(\pi_s)$ is its color sequence or \textit{state}.\\\\
In the remainder of this paper, an order is a binary relation which is reflexive, anti-symmetric and transitive. For any order $\preceq$, one can associate a unique strict order $\prec$ such that $x\prec y$ is equivalent to $x\neq y$ and $x\preceq y$. An order is said to be total if any pair of element can be compared. By abuse of terminology, a strict total order is the strict order associated to a total order.  
\begin{rem}
The binary relation is not necessarily an order. When $\gg$ is a strict total order, we can easily check 
that every finite set of colored parts defines a generalized colored partition, by ordering the parts. In the same way, for a total order, the generalized colored partitions are finite multi-sets of colored integers.
\end{rem}
For example, if we set $\C=\{c\}$ to be a singleton, and the relation $\gg$ defined by
\[k_c\gg l_c \Longleftrightarrow k\geq l\,,\]
one can then see the classical partitions as the generalized partitions $(\pi_1,\ldots,\pi_s)$ such that the last particle $\pi_s$ has a positive potential. Using this definition, we can convey the minimal difference conditions in the weighted words as a relation $\gg$ defined on the set of particles.
\bi The main contribution of this paper will consist in viewing the variant of Bressoud's algorithm, used in the generalization of Siladi\'c's theorem in \cite{IK19}, as a process in which we operate energy transfers according to the states involved in the generalized colored partition. This viewpoint then allows us to see the difference conditions defined for the Siladi\'c theorem as some particular allowable differences between the potentials of consecutive particles. By taking a larger family of allowable differences between the potentials of consecutive particles, we generate an infinite family of identities generalizing the previous result on the Siladi\'c theorem.  
\bi    
\subsection{Statement of Results}
Let $\C$ be a set of states, countable or not,  and let $\Pp=\Z_\C$ be the corresponding set of particles. We recall that the \Par $k_c$ is identified by its potential $k$ and its state $c$. In the remainder of this paper, such a particle is called a \textit{primary} particle.
We consider a relation $\succ$ on $\Z_\C$, related to a certain energy (see \Def{def:minerg}), and we then define the set $\mathcal O$  to be the set of generalized colored partitions with relation $\succ$. 
\m We now define the set of secondary states 
 by $\C_2= \{cc': c,c'\in \C\}$, and we note that the secondary states are non-commutative products of two primary states, i.e. $cc'\neq c'c$ for $c\neq c' \in \C$. We extend this definition to degree $d$ for any $d\geq 1$. The set $\C_d$ of states with degree $d$ is the set of all the non-commutative products of $d$ primary states. We then have $\C_1=\C$, and we use the term "secondary" for degree $2$. The weighted words method is said to be \textit{at degree} $d$ if it only involves states with degree at most $d$.
\bi A \textit{secondary} particle with state $cc'$ is then defined to be a \textit{sum of two consecutive} primary particles, in terms of $\succ$, such that the greater particle (to the left of $\succ$) has color $c$ and the smaller particle (to the right of $\succ$) has color $c'$ (see \Def{def:secpar}). We denote by $\mathcal{S}$ the set of secondary particles. Defining a \textit{suitable} relation $\gg$ on the set of primary and secondary particles $\Pp\sqcup\mathcal S$ (see \Def{def:relation2}), we consider the set $\mathcal E$ of generalized colored partitions consisting of primary or secondary particles well-related by $\gg$.
 \begin{rem}
The color sequence (or state) of an element of $\mathcal O$ or $\mathcal E$  is a finite non-commutative product of primary states in $\C$.
 \end{rem}
 The main theorem of this paper then has the following formulation. 
 \begin{theo}\label{theo:main1}
  For any integer $n$ and any finite non-commutative product  $C$ of colors in $\C$,
 there exists a bijection between $\{\la \in \mathcal O: (C(\la),|\la|)=(C,n)\}$ and $\{\nu \in \mathcal E: (C(\nu),|\nu|)=(C,n)\}$. 
 \end{theo}
 \bi
 An explicit statement of the latter theorem is given in \Thm{theo:degree2}.
 For now, we give an example that will generalize Siladi\'c's theorem to overpartitions. Recall that an overpartition is a partition where we can over-line at most one occurrence of each positive integer \cite{CL04}. It has been a recurrent problem in partition theory to extend some partition identities to overpartitions \cite{DS17,DS14,DS16,DL19,L05,L17}.
\bi
Consider the set of colors $\C=\{\overline{b}<\overline{a}<a<b\}$ and the relation $\succ$ defined by the minimal difference conditions in the following matrix
 \begin{equation}
 D:=\quad\bordermatrix
 {\text{}&\overline{b}&\overline{a}&a&b 
 \cr \overline{b} &1&1&1&1 \cr \overline{a} &0&1&1&1  \cr a &0&0&0&1 \cr b &0&0&0&0}
 \end{equation}
These difference conditions imply that a partition in $\mathcal O$ can have any number of primary particles with a fixed potential and a non over-lined state, while there is at most one primary particle with a  fixed potential and an over-lined state. The partitions of $\mathcal O$ are then identified as the \textit{generalized overpartitions} whose definition is given by the following.
\begin{deff}
Let us fix a set of states $\C$. A generalized overpartition is a generalized partition where we are allowed to over-line at most one particle with a fixed potential and state.
\end{deff}
\begin{ex} The generalized partition $(1_a,1_{\overline{a}},1_{\overline{b}},0_b,0_b,0_a,0_a,0_{\overline{a}},0_{\overline{b}},-1_b,-1_{\overline{a}})$  belongs to $\mathcal O$, and corresponds to the generalized overpartition
$(1_a,\overline{1}_a,\overline{1}_b,0_b,0_b,0_a,0_a,\overline{0}_a,\overline{0}_b,-1_b,\overline{-1}_a)$.
\end{ex}
We then call the partitions in $\mathcal O$ the colored overpartitions, and this means that we can have any number of particles with a fixed potential and state, with at most one such particle over-lined. We observe that once a particle is over-lined, by the difference conditions in $D$, it no longer has the same order with respect to the other particles. For example, we have $1_b\succ 1_a$ but $\overline{1}_b\prec 1_a$. This is different from the usual convention, but the way we defined these relative orders plays a major role in the definition of the corresponding secondary particles.
\bi
We now define the relation $\gg$ by the minimal difference conditions in the following matrix
\begin{equation}\label{eq:diffover}
\begin{tiny}
D' := \quad \bordermatrix{
\text{}&\overline{b}&\overline{a}&a&b&\overline{b}^2&\overline{b}\overline{a}&\overline{b}a&\overline{b}b&\overline{a}\overline{b}&\overline{a}^2&\overline{a}a&\overline{a}b&a\overline{b}&a\overline{a}&a^2&ab&b\overline{b}&b\overline{a}&ba&b^2
\cr \overline{b}&2&2&2&2&2&2&2&2&1&2&2&2&1&1&1&2&1&1&1&1
\cr \overline{a}&1&2&2&2&1&1&1&1&1&2&2&2&1&1&1&2&1&1&1&1
\cr a&1&1&1&2&1&1&1&1&0&1&1&1&0&0&0&1&1&1&1&1
\cr b&1&1&1&1&1&1&1&1&0&1&1&1&0&0&0&1&0&0&0&0
\cr\overline{b}\overline{b}&3&3&3&3&4&4&4&4&3&4&4&4&3&3&3&4&3&3&3&3
\cr\overline{b}\overline{a}&2&3&3&3&2&2&2&2&3&4&4&4&3&3&3&4&3&3&3&3
\cr\overline{b}a&2&2&2&3&2&2&2&2&1&2&2&2&1&1&1&2&3&3&3&3
\cr\overline{b}b&2&2&2&2&2&2&2&2&1&2&2&2&1&1&1&2&1&1&1&1
\cr\overline{a}\overline{b}&1&1&1&1&3&3&3&3&2&3&3&3&2&2&2&3&2&2&2&2
\cr\overline{a}^2&2&3&3&3&2&2&2&2&3&4&4&4&3&3&3&4&3&3&3&3
\cr\overline{a}a&2&2&2&3&2&2&2&2&1&2&2&2&1&1&1&2&3&3&3&3
\cr\overline{a}b&2&2&2&2&2&2&2&2&1&2&2&2&1&1&1&2&1&1&1&1
\cr a\overline{b}&1&1&1&1&3&3&3&3&2&3&3&3&2&2&2&3&2&2&2&2
\cr a\overline{a}&1&2&2&2&1&1&1&1&2&3&3&3&2&2&2&3&2&2&2&2
\cr a^2&1&1&1&2&1&1&1&1&0&1&1&1&0&0&0&1&2&2&2&2
\cr ab&2&2&2&2&2&2&2&2&1&2&2&2&1&1&1&2&1&1&1&1
\cr b\overline{b}&1&1&1&1&3&3&3&3&2&3&3&3&2&2&2&3&2&2&2&2
\cr b\overline{a}&1&2&2&2&1&1&1&1&2&3&3&3&2&2&2&3&2&2&2&2
\cr ba&1&1&1&2&1&1&1&1&0&1&1&1&0&0&0&1&2&2&2&2
\cr b^2&1&1&1&1&1&1&1&1&0&1&1&1&0&0&0&1&0&0&0&0
}
\end{tiny}
\end{equation}
 By definition, the secondary particles with state $cc'$ then have a potential with the same parity as the entry of $D$ corresponding to the line $c$ and the column $c'$. Therefore, we have the following correspondence for secondary states:
\begin{equation}\label{eq:overcolor}
\bordermatrix{\text{} &\overline{b}&\overline{a}&a&b 
 \cr \overline{b} &b^2_{odd}&ba_{odd}&\overline{ba}_{odd}&\overline{b}^2_{odd} \cr \overline{a} &ab_{even}&a^2_{odd}&\overline{a}^2_{odd}&\overline{ab}_{odd} \cr a &\overline{ab}_{even}&\overline{a}^2_{even}&a^2_{even}&ab_{odd} \cr b &\overline{b}^2_{even}&\overline{ba}_{even}&ba_{even}&b^2_{even}}
\,, 
\end{equation}
where $c_{parity}$ refers to a particle with state $c$ and potential with the same parity as the index $parity$.
Here again, the generalized partitions in $\mathcal E$ can be identified as some generalized overpartitions for the set of colors $\{a,b,a^2,ab,ba,b^2\}$.  
We now state the corresponding corollary to \Thm{theo:main1}. To simplify the formulation of the corollary, we assume that the symbols $a,b$ and $c$ commute in the generating functions.  
\begin{cor}\label{cor:2col}
Let $u,v,w$ and $n$ be non-negative integers. Let us denote by 
$A(n;u,v,w)$ the number of colored overpartitions of size $n$ with positive potentials and colors in $\{a,b\}$, with  
$u$ particles with color $a$, $v$ particles with color $b$ and 
$w$ over-lined particles.
Let us denote by 
$B(n;u,v,w)$ the number of colored overpartitions of size $n$ with colors in $\{a,b,a^2,ab,ba,b^2\}$, with positive potential for the primary particles and potential greater than one for the secondary particles, satisfying the minimal difference conditions given by $D'$, with
$u$ occurrences of the symbol $a$, $v$ occurrences of the symbol $b$, and such that $w$ equals the number of over-lined particles plus twice the number of even particles with color $ab$ and odd particles with color $a^2,ba$ or $b^2$.
We then have $A(n;u,v,w)=B(n;u,v,w)$ and the identity
\begin{equation}
\sum_{n,u,v,w\geq 0} B(n;u,v,w) a^u b^v c^w d^{u+v-w} q^n = \sum_{n,u,v,w\geq 0} A(n;u,v,w) a^u b^v c^w  d^{u+v-w}q^n=\frac{(-acq;q)_\infty(-bcq;q)_\infty}{(adq;q)_\infty(bdq;q)_\infty}\,\cdot 
\end{equation}
\end{cor}
In the previous corollary, if we restrict the partitions in $\mathcal O$ to those with only over-lined particles, i.e. $u+v=w$, and by applying the transformations $(q,a,b,c,d)\mapsto(q^4,q^{-1},q^{-3},1,0)$, we retrieve the identity given by Siladi\'c in \Thm{theo:siladic}.
\m On the other hand, by restricting the partitions in $\mathcal O$ to those with only non over-lined particles, i.e. $w=0$, and by applying the transformations $(q,a,b,c,d)\mapsto(q^4,q^{-3},q^{-1},0,1)$, we obtain the following analogous theorem of Siladi\'c's identity.
\begin{theo}\label{theo:varsiladic}
The number of partitions $\lambda_1+\cdots+\lambda_s$ of an integer $n$ into odd parts 
is equal to the number of partitions of $n$ such that
\begin{align*}
\lambda_i -\lambda_{i+1}  = 0 &\Rightarrow \, \lambda_i+\lambda_{i+1}  \equiv \pm 4 \mod 16\,,\\
\lambda_i -\lambda_{i+1}  = 1 &\Rightarrow \, \lambda_i +\lambda_{i+1} \equiv \pm 3 \mod 16\,,\\
\lambda_i -\lambda_{i+1}  = 2 &\Rightarrow \, \lambda_i +\lambda_{i+1} \equiv \pm 2,\pm 6 \mod 16\,,\\
\lambda_i -\lambda_{i+1}  = 3 &\Rightarrow \, \lambda_i +\lambda_{i+1} \equiv \pm 1,\pm 5,\pm 7\mod 16\,\,\cdot
\end{align*}
\end{theo}
\begin{ex}
For $n=10$, the partitions of $n$ into odd parts are
\[(9,1),(7,3),(7,1,1,1),(5,5),(5,3,1,1),(5,1,1,1,1,1),(3,3,3,1),(3,3,1,1,1,1)\]
\[(3,1,1,1,1,1,1,1)\text{ and } (1,1,1,1,1,1,1,1,1,1)\]
and the partitions of given by \Thm{theo:varsiladic} are
\[(10),(9,1),(8,2),(7,3),(7,2,1),(6,4),(6,2,2),(5,2,2,1),(4,2,2,2)\text{ and } (2,2,2,2,2)\,\cdot\]  
\end{ex}
\begin{rem}
For Siladi\'c's theorem, since we have $\overline{b}<\overline{a}$, we do the transformation $(a,b)\mapsto (q^{-1},q^{-3})$ to keep the order, while for the analogous theorem, we have $a<b$ and we then apply $(a,b)\mapsto (q^{-3},q^{-1})$.
\end{rem}
\bi 
The remainder of the paper is organized as follows. We first present in \Sct{sect:setup} the key tools and state explicitly the main result of this paper, namely \Thm{theo:degree2}. Then in \Sct{sect:maps}, we give the two bijections for \Thm{theo:degree2} which are inverse to each other. After that, in \Sct{sect:proof}, we prove the well-definedness of the bijections. Finally, in \Sct{sect:remarks}, we close with some remarks and we make the connection with the second part of this series of papers.
\section{The setup}\label{sect:setup}
Let $\C$ be a set of states, countable or not. We recall the set of primary particles $\Z_\C$, which we also denote by $\Pp=\Z \times \C$. In the following, a primary particle with potential $k$ and state $c$ is identified as $k_c$ or $(k,c)$.
\begin{deff}\label{def:minerg}
A \textit{minimal energy} is a function $\ep$  from $\C^2$ to $\{0,1\}$. The term minimal here refers to energies with values in $\Z_{\geq 0}$, as $0$ and $1$ are the smallest non-negative integers. When $\C=\{c_1,\ldots,c_n\}$ is a finite set, the data given by $\ep$ is equivalent to the matrix $M_\ep = (\ep(c_i,c_j))_{i,j=1}^n$, which we call the \textit{energy matrix} for $\ep$.
\m
We say that a minimal energy is \textit{transitive} if 
it satisfies the \Ti:
\begin{equation}
 \forall c,c',c''\in \C\,, \quad \ep(c,c'')\leq \ep(c,c')+\ep(c',c'')\,\cdot
\end{equation}
Let $c_1,\ldots,c_t$ be a sequence of primary states. We then define the \textit{energy of transfer} from $c_1$ to $c_t$ to be the sum of the intermediate minimal energies:
\begin{equation}\label{eq:entrans}
\sum_{i=1}^{t-1} \ep(c_i,c_{i+1})\,\cdot
\end{equation}
\end{deff}
\begin{rem}
Note that if $\ep$ is a (transitive) minimal energy, then $\ep^* : (c,c') \mapsto \ep(c',c)$ is also a (transitive) minimal energy. Furthermore, if $\C$ is finite, the energy matrix $M_{\ep^*}$ is then the transpose of the energy matrix $M_\ep$.
\end{rem}
In the remainder of this paper, we consider $\ep$ to be a minimal energy.
\begin{deff}
The \textit{energy relation} $\od$ with respect to $\ep$ is the binary relation on $\Pp^2$ defined by
\begin{equation}\label{rel}
 (k,c)\od (k',c') \Longleftrightarrow \, k-k'\geq \ep(c,c')\,\cdot
\end{equation} 
This relation is transitive if and only if $\ep$ is transitive.
\end{deff}
\begin{exs}\label{ex:tri}
Let $\C=\{c_1,\ldots,c_n\}$ be a set of states.
For any proposition $A$, set $\chi(A)=1$ if $prop$ is true and $\chi(A)=0$ otherwise. 
\begin{enumerate}
 \item For $\ep(c_i,c_j) = \chi(i<j)$, we can set on $\C$ the strict order $c_1<\cdots<c_n$ and the energy relation $\od$ becomes the lexicographic order on $\Pp$: 
 \[\cdots\od (k+1)_{c_1}\od k_{c}\od k_{c_n}\od k_{c_{n-1}}\od k_{c_{n-1}}\od\cdots \od k_{c_2} \od k_{c_2}\od k_{c_1}\od k_{c_1}\od \cdots\,\cdot\]
 Here, ordering $k_{c_i}\od k_{c_i}$ simply indicates a possible repetition of the part $k_{c_i}$. 
 The corresponding energy matrix is given by
 \[M_\ep = 
 \bordermatrix{
 \text{}&c_1&c_2&\cdots&c_{n-1}&c_n
 \cr c_1 &0&1&\cdots&1&1
 \cr c_2&0&0&\ddots&1&1
 \cr \vdots&\vdots&\vdots&\ddots&\ddots&\vdots
 \cr c_{n-1}&0&0&\cdots&0&1
 \cr c_n&0&0&\cdots&0&0
 }\,\cdot\]
 \item For $\ep(c_i,c_j) = \chi(i\leq j)$, using the previous ordering on $\C$, the energy relation $\od$ is the strict lexicographic order on $\Pp$:
 \[\cdots\od (k+1)_{c_1}\od k_{c_n}\od k_{c_{n-1}}\od \cdots\od k_{c_2}\od k_{c_1}\od \cdots\,\cdot\]
 Here, we do not have a repetition of the parts $k_{c_i}$. 
 The corresponding energy matrix is given by
 \[M_\ep = 
 \bordermatrix{
 \text{}&c_1&c_2&\cdots&c_{n-1}&c_n
 \cr c_1 &1&1&\cdots&1&1
 \cr c_2&0&1&\cdots&1&1
 \cr \vdots&\vdots&\ddots&\ddots&\vdots&\vdots
 \cr c_{n-1}&0&0&\ddots&1&1
 \cr c_n&0&0&\cdots&0&1
 }\,\cdot\]
\end{enumerate}
\end{exs}
\begin{ex}\label{ex:over}
Let $\C'=\{c_1,\cdots,c_n\}$  be a set of states. If we set $\overline{\C}'=\{\overline{c}:c\in \C'\}$ and $\C=\C'\sqcup \overline{C}'$ we can then define $\ep$ on $\C^2$, for any $i,j\in \{1,\cdots,n\}$, by the following:
 \begin{enumerate}
  \item $\ep(c_i,c_j) = \chi(i<j)\,,$
  \item $\ep(c_i,\cc_j) = 0\,,\,\ep(\cc_i,c_j) = 1\,,$
  \item $\ep(\cc_i,\cc_j)=\chi(i\geq j)\,\cdot$
 \end{enumerate}
The relation $\od$ is then an order on $\Z_{\C}$, where over-lined colored particles can occur at most once in any ordered chain:
\[\cdots\od (k+1)_{\cc_n}\od k_{c_n}\od k_{c_n}\od k_{c_{n-1}}\od\cdots \od k_{c_2}\od k_{c_1}\od k_{c_1}\od  k_{\cc_1}\od k_{\cc_2}\od \cdots k_{\cc_{n-1}}\od k_{\cc_n} \od \cdots\,\cdot \]
The latter inequalities give some generalized colored partitions that can be identified as overpartitions.
The corresponding energy matrix is given by
 \[M_\ep = 
 \bordermatrix{
 \text{}&\cc_n&\cdots&\cc_1&c_1&\cdots&c_n
 \cr \cc_n &1&\cdots&1&1&\cdots&1
 \cr \vdots&\vdots&\ddots&\vdots&\vdots&1^{\star}&\vdots
 \cr \cc_1&0&\cdots&1&1&\cdots&1
 \cr c_1&0&\cdots&0&0&\cdots&1
 \cr \cdots& \vdots&0^{\star}&\vdots&\vdots&\ddots&\vdots
 \cr c_n&0&\cdots&0&0&\cdots&0
 }\,\cdot\]
 Note that the examples given in \Expls{ex:tri} respectively correspond to the restriction to $\{c_1,\ldots,c_n\}$ in the first case, and the restriction to $\{\cc_n,\ldots,\cc_1\}$, with $\cc_i\equiv c_{n+1-i}$ in the second case.
\m We also remark that for $\C'=\{a<b\}$, we retrieve the primary particles used in \Cor{cor:2col}.  
\end{ex}
\begin{ex}\label{ex:twister}
Let us consider $\C=\{a,b\}$, and the minimal energy $\ep$ given by the following energy matrix:
\[
M_\ep = \bordermatrix{
\text{}&a&b
\cr a &1&0
\cr b &0&1
}\,\cdot\] 
The relation $\ep$ is not transitive, as we have $\ep(a,a)>\ep(a,b)+\ep(b,a)$. The well-ordered sequences of particles with the same potential have the form
\[\cdots \od k_a\od k_b\od k_a\od k_b\od \cdots\,\cdot\]
\end{ex} 
\bi We recall that a secondary state is the product of two primary states. The key idea is to build secondary particles starting from the primary particles. The following definition permits a suitable construction for these secondary particles.
\begin{deff}\label{def:secpar}
We define the \textit{secondary particles} as sums of two consecutive primary particles in terms of $\od$. We denote by $\Sc = \Z\times \C^2$ the set of secondary particles, in such a way that the particle
\begin{equation}
 (k,c,c') = (k+\ep(c,c'),c)+(k,c') 
\end{equation}
has potential $2k+\ep(c,c')$ and state $cc'$. In fact, $(k+\ep(c,c'),c)$ is exactly the primary particle  of state $c$ with smallest potential, which is well-related to $(k,c')$ in terms of $\od$.
We then set the functions $\gamma$ and $\mu$ on $\Sc$, defined by
\begin{equation}
 \gamma(k,c,c') = (k+\ep(c,c'),c)\text{ and }\mu(k,c,c') = (k,c')\,,
\end{equation} 
to be respectively the \textit{upper} and \textit{lower} halves of $(k,c,c')$. In the following, we identify  a secondary particle as $(k,c,c')$ or $(2k+\ep(c,c'))_{cc'}$.
\end{deff}
\begin{ex}
Let us take  $\C=\{a,\overline{a}\}$ in \Expl{ex:over}. We then have 
\[
\bordermatrix{
\text{}&\overline{a}&a
\cr \overline{a}& 1&1
\cr a&0&0
}
\,\]
and we obtain with \Def{def:secpar} and \eqref{eq:overcolor} the following secondary particles:
\[\begin{cases}
(k,a,a) = 2k_{a^2}\,,\\
(k,a,\overline{a}) = 2k_{a\overline{a}}\equiv \overline{2k}_{a^2}\,,
\\(k,\overline{a},a) = 2k+1_{\overline{a}a} \equiv \overline{2k+1}_{a^2}\,,
\\(k,\overline{a},\overline{a}) = 2k+1_{\overline{a}^2}\equiv 2k+1_{a^2}\,\cdot
\end{cases}\]   
\end{ex}
We now build a relation on the set $\Pp\sqcup \Sc$ of primary and secondary particles.
\begin{deff}\label{def:relation2}
We define the relation $\odp$ on $\Pp\sqcup \Sc$ as follows: 
\begin{enumerate}
 \item Two primary particles of $\Pp$ are well-ordered by $\odp$ if and only if they are well-ordered but not consecutive in terms of $\od$: 
 \begin{equation}\label{pp}
(k,\textcolor{red}{c})\odp (k',\textcolor{foge}{c'}) \Longleftrightarrow k-k' > \ep(\textcolor{red}{c},\textcolor{foge}{c'})\,\cdot
\end{equation}
\item A primary particle of $\Pp$ is well-ordered with a secondary particle of $\Sc$ if and only if their potentials' difference  is at least equal to the energy of transfer from the first to the last primary states: 
\begin{equation}\label{ps}
(k,\textcolor{red}{c})\odp(k',\textcolor{foge}{c'},\textcolor{blue}{c''}) \Longleftrightarrow k - (2k'+\ep(\textcolor{foge}{c'},\textcolor{blue}{c''}))\geq \ep(\textcolor{red}{c},\textcolor{foge}{c'})+\ep(\textcolor{foge}{c'},\textcolor{blue}{c''})\,\cdot
\end{equation}
\item A secondary particle of $\Sc$ is well-ordered with a primary particle of $\Pp$ if and only if their potentials' difference  is greater than the transfer energy (from first to last state): 
\begin{equation}\label{sp}
(k,\textcolor{red}{c},\textcolor{foge}{c'})\odp (k',\textcolor{blue}{c''})\Longleftrightarrow (2k+\ep(\textcolor{red}{c},\textcolor{foge}{c'}))-k' > \ep(\textcolor{red}{c},\textcolor{foge}{c'})+\ep(\textcolor{foge}{c'},\textcolor{blue}{c''})\,\cdot
\end{equation}
\item Two secondary particles of $\Sc$ are well-ordered by $\odp$ if and only if the lower half of the first one is greater than the upper half of the second in terms of $\od$:
\begin{equation}\label{ss}
(k,\textcolor{red}{c},\textcolor{foge}{c'})\odp (k',\textcolor{blue}{c''},\textcolor{purple}{c'''}) \Longleftrightarrow \mu(k,\textcolor{red}{c},\textcolor{foge}{c'})\od \gamma(k',\textcolor{blue}{c''},\textcolor{purple}{c'''})\,\cdot
\end{equation}
This is equivalent to saying that the potentials' difference $k-k'$ is at least equal to the energy of transfer $\ep(\textcolor{foge}{c'},\textcolor{blue}{c''})+\ep(\textcolor{blue}{c''},\textcolor{purple}{c'''})$.
\end{enumerate}
\end{deff}
One can check that for $\C'=\{a<b\}$ and the minimal energy $\ep$ described in \Expl{ex:over}, the relations in the latter definition exactly give the minimal difference conditions presented in \eqref{eq:diffover}. 
\begin{rem}
We notice that 
\begin{equation}\label{cons}
 (k,c)\od(k',c')\text{ and }(k,c)\not\odp(k',c')\Longleftrightarrow k-k'=\ep(c,c')\,\cdot
\end{equation}
Such pair of primary particles is called a troublesome pair.
\end{rem}
\begin{deff}\label{def:rho}
We define $\Odd$ (respectively $\E$) to be the set of all generalized colored partitions with particles in $\Pp$ (respectively $\Pp\sqcup\Sc$)
and relation $\od$ (respectively $\odp$). 
\bi
For $\rho\in\{0,1\}$, we consider the following sets:
\begin{itemize}
 \item $\Pp^{\rho_+} =\Z_{\geq \rho}\times\C \text{  and  } \Sc^{\rho_+} = \Z_{\geq \rho}\times\C^2 =\{(k,c,c')\in \Sc:\,k\geq \rho\}$,
 \item $\Pp^{\rho_-} =\Z_{\leq \rho}\times\C \text{  and  } \Sc^{\rho_-} =\{(k,c,c')\in \Sc:\,k+\ep(c,c')\leq \rho\}$.
\end{itemize}
We then denote by $\Odd^{\rho_+}$ (respectively $\Odd^{\rho_-}$) the subset of $\Odd$ of generalized colored partitions with particles in $\Pp^{\rho_+}$ (respectively $\Pp^{\rho_-}$), and by 
$\E^{\rho_+}$ (respectively $\E^{\rho_-}$) the subset of $\E$ of generalized colored partitions with particles in $\Pp^{\rho_+}\sqcup \Sc^{\rho_+}$ (respectively $\Pp^{\rho_-}\sqcup \Sc^{\rho_-}$).
\bi
Since the secondary states are products of two primary states, the states of partitions in $\Odd$ and $\E$ are then seen as a finite \textit{non-commutative} product of primary states in $\C$.
\end{deff}
We now present the main result of this paper.
\begin{theo}\label{theo:degree2}
 For any integer $n$ and any state $C$ as a finite non-commutative product of states in $\C$,
 there exists a bijection between $\{\la \in \Odd: (C(\la),|\la|)=(C,n)\}$ and $\{\nu \in \E: (C(\nu),|\nu|)=(C,n)\}$. 
 In particular, for $\rho \in \{0,1\}$, we have the identities
 \begin{align}
 |\{\nu \in \E^{\rho_+}: (C(\nu),|\nu|)=(C,n)\}| &= |\{\la \in \Odd^{\rho_+}: (C(\la),|\la|)=(C,n)\}|\,,\\
 |\{\nu \in \E^{\rho_-}: (C(\nu),|\nu|)=(C,n)\}| &= |\{\la \in \Odd^{\rho_-}: (C(\la),|\la|)=(C,n)\}|\,\cdot
 \end{align}
\end{theo}
One can observe that, for any integer $n$ and any state $C$ with at least two primary states, the sets
$\{\la \in \Odd: (C(\la),|\la|)=(C,n)\}$ and  $\{\nu \in \E: (C(\nu),|\nu|)=(C,n)\}$ are infinite. However, as soon as we give an upper or a lower bound on the particles' potentials, the corresponding subsets are finite. 
\begin{ex}Let us consider $\C'=\{a<b\}$ in \Expl{ex:over} and the corresponding minimal energy. We then have for $n=10$ and $C=\overline{b}\overline{a}ba$ the relation  $\{\la \in \Odd^{1_{-}}: (C(\la),|\la|)=(\overline{b}\overline{a}ba,10)\}=\{\la \in \E^{1_{-}}: (C(\la),|\la|)=(\overline{b}\overline{a}ba,10)\}=\emptyset$, and the partitions in  $\Odd^{\rho_{+}}$ and $\E^{\rho_{+}}$ with the corresponding energy and state are given in the following table:
\[\begin{array}{|c|c|c|c|}
\hline
\Odd^{0_+}&\Odd^{1_+}&\E^{0_+}&\E^{1_+}\\
\hline
(9_{\overline{b}},1_{\overline{a}},0_b,0_a)&&(9_{\overline{b}},1_{\overline{a}},0_{ba})&\\
(8_{\overline{b}},2_{\overline{a}},0_b,0_a)&&(8_{\overline{b}},2_{\overline{a}},0_{ba})&\\
(7_{\overline{b}},3_{\overline{a}},0_b,0_a)&&(7_{\overline{b}},3_{\overline{a}},0_{ba})&\\
(7_{\overline{b}},2_{\overline{a}},1_b,0_a)&&(7_{\overline{b}},3_{\overline{a}b},0_a)&\\
(6_{\overline{b}},4_{\overline{a}},0_b,0_a)&&(6_{\overline{b}},4_{\overline{a}},0_{ba})&\\
(6_{\overline{b}},3_{\overline{a}},1_b,0_a)&&(6_{\overline{b}},3_{\overline{a}},1_b,0_a)&\\
(6_{\overline{b}},2_{\overline{a}},1_b,1_a)&(6_{\overline{b}},2_{\overline{a}},1_b,1_a)&(6_{\overline{b}},3_{\overline{a}b},1_a)&(6_{\overline{b}},3_{\overline{a}b},1_a)\\
(5_{\overline{b}},4_{\overline{a}},1_b,0_a)&&(9_{\overline{b}\overline{a}},1_b,0_a)&\\
(5_{\overline{b}},3_{\overline{a}},2_b,0_a)&&(7_{\overline{b}\overline{a}},3_b,0_a)&\\
(5_{\overline{b}},3_{\overline{a}},1_b,1_a)&(5_{\overline{b}},3_{\overline{a}},1_b,1_a)&(5_{\overline{b}},3_{\overline{a}},2_{ba})&(5_{\overline{b}},3_{\overline{a}},2_{ba})\\
(4_{\overline{b}},3_{\overline{a}},2_b,1_a)&(4_{\overline{b}},3_{\overline{a}},2_b,1_a)&(7_{\overline{b}\overline{a}},2_b,1_a)&(7_{\overline{b}\overline{a}},2_b,1_a)\\
\hline
\end{array} \,\cdot
\]
We have for $n=-8$ and $C=\overline{b}\overline{a}ba$ the relation $\{\la \in \Odd^{0_{+}}: (C(\la),|\la|)=(\overline{b}\overline{a}ba,-8)\}=\{\la \in \E^{0_{+}}: (C(\la),|\la|)=(\overline{b}\overline{a}ba,-8)\}=\emptyset$ and the partitions  in  $\Odd^{\rho_{-}}$ and $\E^{\rho_{-}}$ with the corresponding energy and state are given in the following table:
\[\begin{array}{|c|c|c|c|}
\hline
\Odd^{1_-}&\Odd^{0_-}&\E^{1_-}&\E^{0_-}\\
\hline
(1_{\overline{b}},0_{\overline{a}},-1_b,-8_a)&&(1_{\overline{b}},-1_{\overline{a}b},-8_a)&\\
(1_{\overline{b}},0_{\overline{a}},-2_b,-7_a)&&(1_{\overline{b}\overline{a}},-2_b,-7_a)&\\
(1_{\overline{b}},0_{\overline{a}},-3_b,-6_a)&&(1_{\overline{b}\overline{a}},-3_b,-6_a)&\\
(1_{\overline{b}},-1_{\overline{a}},-2_b,-6_a)&&(1_{\overline{b}},-3_{\overline{a}b},-8_a)&\\
(1_{\overline{b}},0_{\overline{a}},-4_b,-5_a)&&(1_{\overline{b}\overline{a}},-4_b,-5_a)&\\
(1_{\overline{b}},-1_{\overline{a}},-3_b,-5_a)&&(1_{\overline{b}},-1_{\overline{a}},-3_b,-5_a)&\\
(0_{\overline{b}},-1_{\overline{a}},-2_b,-5_a)&(0_{\overline{b}},-1_{\overline{a}},-2_b,-5_a)&(0_{\overline{b}}
,-3_{\overline{a}b},-5_a)&(0_{\overline{b}}
,-3_{\overline{a}b},-5_a)\\
(1_{\overline{b}},-1_{\overline{a}},-4_b,-4_a)&&(1_{\overline{b}},-1_{\overline{a}},-8_{ba})&\\
(1_{\overline{b}},-2_{\overline{a}},-3_b,-4_a)&&(1_{\overline{b}},-3_{\overline{a}},-6_{ba})&\\
(0_{\overline{b}},-1_{\overline{a}},-3_b,-4_a)&(0_{\overline{b}},-1_{\overline{a}},-3_b,-4_a)&
(-1_{\overline{b}},-3_{\overline{a}},-4_{ba})&(-1_{\overline{b}},-3_{\overline{a}},-4_{ba})\\
(0_{\overline{b}},-2_{\overline{a}},-3_b,-3_a)&(0_{\overline{b}},-2_{\overline{a}},-3_b,-3_a)&
(0_{\overline{b}},-2_{\overline{a}},-6_{ba})&(0_{\overline{b}},-2_{\overline{a}},-6_{ba})\\
\hline
\end{array} 
\]
\end{ex}
\bi  
We obtain the following corollary of \Thm{theo:degree2}.
\begin{cor}
For any set $\C$ of primary states and any minimal energy $\ep$ on $\C^2$, we have 
\begin{equation}
\sum_{\substack{n\geq 0\\
C \in <\C>}}|\{\nu \in \E^{\rho_+}: (C(\nu),|\nu|)=(C,n)\}|\underline{C}q^n = 
\sum_{\substack{n\geq 0\\
C \in <\C>}}|\{\la \in \Odd^{\rho_+}: (C(\la),|\la|)=(C,n)\}|\underline{C}q^n = \prod_{m\geq \rho} F_{\C}(\ep;q^m)
\end{equation}
where $<\C>$ is the non-commutative monoid generated by the primary states of $\C$, and $F_{\C}(\ep,x)$, in the commutative algebra $\Z[[\C,x]]$, is the generating function of all the partitions in $\Odd$ with particles' potential equal to $1$, and $\underline{C}$ is the commutative product corresponding to $C$ in $\Z[[\C,x]]$.  In particular, we have the following explicit expressions for $F_{\C}(\ep,x)$: 
\begin{enumerate}
\item For $\C=\{c_1,\ldots,c_n\}$, we have 
\begin{equation}
\begin{array}{|c|c|}
\hline
\ep(c_i,c_j)& F_{\C}(\ep,x)\\
\hline
\hline
0& \displaystyle\frac{1}{1-(c_1+\cdots+c_n)x} \\
\hline
1& \displaystyle 1+(c_1+\cdots+c_n)x \\
\hline
\chi(i\neq j)& \displaystyle 1+ \sum_{i=1}^n \frac{c_ix}{1-c_ix} \\
\hline
\chi(i<j)& \displaystyle \prod_{i=1}^n \frac{1}{1-c_ix}\\
\hline 
\chi(i\leq j)& \displaystyle \prod_{i=1}^n (1+c_ix)\\
\hline
\end{array}
\end{equation}
\item For $\C'=\{c_1,\ldots,c_n\}$ and $\ep$ as described in \Expl{ex:over}, 
\begin{equation}
F_{\C}(\ep,x) = \prod_{i=1}^n \frac{1+\cc_ix}{1-c_ix}\,\cdot
\end{equation}
\item For $\C=\{a,b\}$ and $\ep$ as described in \Expl{ex:twister}, 
\begin{equation}
F_{\C}(\ep,x) = \frac{(1+ax)(1+bx)}{(1-abx^2)}\,\cdot
\end{equation}
\end{enumerate}
\end{cor}
\bi 
\begin{rem}
In \Thm{theo:degree2}, by setting $\C = \{a_1<\cdots<a_t\}$ and $\ep(a_i,a_j) = \chi(i\leq j)$, and $\rho_+=1_+$ (equivalent to the last case of (1) in the latter corollary), we recover the main generalization of Siladi\'c's theorem given in \cite{IK19}.
\end{rem}
The remainder of the paper will focus on the bijective proof of \Thm{theo:degree2}.
\section{Bijective maps for \Thm{theo:degree2}}\label{sect:maps}
In this section, we define an operator on the pairs of particles of different degree (primary and secondary), presented as an energy transfer, and a bijection for the proof of \Thm{theo:degree2} which uses this operator.
\subsection{Energy transfer}
\begin{deff}
We define a mapping $\Lambda$ on $\Pp\times \Sc \sqcup \Sc\times \Pp$ by the following relations:
\begin{equation}\label{eq:lmbps}
\begin{array}{r c l}
 \Pp\times \Sc &\longrightarrow&\Sc\times \Pp\\
 (k,\textcolor{red}{c}),(k',\textcolor{foge}{c'},\textcolor{blue}{c''}) &\longmapsto&(k'+\ep(\textcolor{foge}{c'},\textcolor{blue}{c''}),\textcolor{red}{c},\textcolor{foge}{c'}),(k-\ep(\textcolor{red}{c},\textcolor{foge}{c'})-\ep(\textcolor{foge}{c'},\textcolor{blue}{c''}),\textcolor{blue}{c''})\end{array}\,,
\end{equation}
\begin{equation}\label{eq:lmbsp}
\begin{array}{rcl}
\Sc\times \Pp&\longrightarrow&\Pp\times \Sc\\
 (k,\textcolor{red}{c},\textcolor{foge}{c'}),(k',\textcolor{blue}{c''})& \longmapsto &(k'+\ep(\textcolor{red}{c},\textcolor{foge}{c'})+\ep(\textcolor{foge}{c'},\textcolor{blue}{c''}),\textcolor{red}{c}),(k-\ep(\textcolor{foge}{c'},\textcolor{blue}{c''}),\textcolor{foge}{c'},\textcolor{blue}{c''})
\end{array}\,\cdot
\end{equation}
\end{deff}
What does $\Lambda$ do to the particles? Let us consider the following diagrams according to the occurrences of primary states:
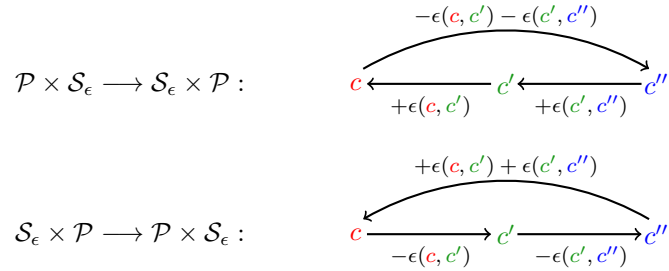
\begin{center}
\begin{tikzpicture}
\draw (-3,0) node {$ \Pp\times \Sc \longrightarrow\Sc\times \Pp :$};
\draw (0,0) node {$\textcolor{red}{c}$};
\draw (2,0) node {$\textcolor{foge}{c'}$};
\draw (4,0) node {$\textcolor{blue}{c''}$};
\foreach \x in {0,1}
\draw [thick, ->] (3.8-2*\x,0)--(2.15-2*\x,0);
\draw [thick, ->] (0.1,0.2) arc (120:60:3.8);
\draw (3,-0.3) node {{\footnotesize $+\ep(\textcolor{foge}{c'},\textcolor{blue}{c''})$}};
\draw (1,-0.3) node {{\footnotesize $+\ep(\textcolor{red}{c},\textcolor{foge}{c'})$}};
\draw (2,0.9) node {{\footnotesize $-\ep(\textcolor{red}{c},\textcolor{foge}{c'})-\ep(\textcolor{foge}{c'},\textcolor{blue}{c''})$}};

\draw (-3,-2) node {$ \Sc\times \Pp \longrightarrow\Pp\times \Sc :$};
\draw (0,-2) node {$\textcolor{red}{c}$};
\draw (2,-2) node {$\textcolor{foge}{c'}$};
\draw (4,-2) node {$\textcolor{blue}{c''}$};
\foreach \x in {0,1}
\draw [thick, <-] (3.8-2*\x,-2)--(2.15-2*\x,-2);
\draw [thick, <-] (0.1,-1.8) arc (120:60:3.8);
\draw (3,-2.3) node {{\footnotesize $-\ep(\textcolor{foge}{c'},\textcolor{blue}{c''})$}};
\draw (1,-2.3) node {{\footnotesize $-\ep(\textcolor{red}{c},\textcolor{foge}{c'})$}};
\draw (2,-1.1) node {{\footnotesize $+\ep(\textcolor{red}{c},\textcolor{foge}{c'})+\ep(\textcolor{foge}{c'},\textcolor{blue}{c''})$}};
\end{tikzpicture}
\end{center}
These diagrams sum up the transfer of energies that occurs during the application of $\Ll$. For example, one can understand the process on the first diagram as follows: 
\begin{enumerate}
 \item The lower half $(k',\textcolor{blue}{c''})$ moves from state $\textcolor{blue}{c''}$ to $\textcolor{foge}{c'}$ and gains the minimal energy $\ep(\textcolor{foge}{c'},\textcolor{blue}{c''})$:
 \[
 \begin{array}{rlc}
  \textcolor{foge}{c'}&\longleftarrow&\textcolor{blue}{c''}\\
  k'+\ep(\textcolor{foge}{c'},\textcolor{blue}{c''})&\longleftarrow&k'
 \end{array}\,\cdot
 \]
 \item The upper half $(k'+\ep(\textcolor{foge}{c'},\textcolor{blue}{c''}),\textcolor{foge}{c'})$ moves from state $\textcolor{foge}{c'}$ to $\textcolor{red}{c}$ and gains the minimal energy $\ep(\textcolor{red}{c},\textcolor{foge}{c'})$:
 \[
 \begin{array}{rcl}
  \textcolor{red}{c}&\longleftarrow&\textcolor{foge}{c'}\\
  k'+\ep(\textcolor{red}{c},\textcolor{foge}{c'})+\ep(\textcolor{foge}{c'},\textcolor{blue}{c''})&\longleftarrow& k'+\ep(\textcolor{foge}{c'},\textcolor{blue}{c''})
 \end{array}\,\cdot
 \]
 \item The primary particle $(k,\textcolor{red}{c})$ moves from state $\textcolor{red}{c}$ to state $\textcolor{blue}{c''}$, through state $\textcolor{foge}{c'}$, and loses the energy of transfer $\ep(\textcolor{red}{c},\textcolor{foge}{c'})+\ep(\textcolor{foge}{c'},\textcolor{blue}{c''})$:
 \[
 \begin{array}{rcccl}
  \textcolor{red}{c}&\longrightarrow&\textcolor{foge}{c'}&\longrightarrow&\textcolor{blue}{c''}\\
  k&\longrightarrow&k-\ep(\textcolor{red}{c},\textcolor{foge}{c'})&\longrightarrow&k-\ep(\textcolor{red}{c},\textcolor{foge}{c'})-\ep(\textcolor{foge}{c'},\textcolor{blue}{c''})
 \end{array}\,\cdot
 \]
\end{enumerate}
The second diagram follows exactly the same transfer of energies.
We can then see $\Lambda$ as a \textit{energy transfer} that conserves the sequence of states but switches particles with the minimal loss or gain of energies.
One can check that the operator $\Lambda$ is an involution, i.e. $\Lambda^2 = Id$.
\bi In the following, if we apply $\Lambda$ to a pair of particles $(x,y)$ in $\Pp\times \Sc \sqcup \Sc\times \Pp$, we say that we \textit{cross} the particles $x$ and $y$.
\begin{ex}
We take $\C'=\{a<b\}$ in \Expl{ex:over}. We then have $\Lambda(3_{ab},-10_{\overline{a}})=(-9_a,2_{b\overline{a}})$. The energy transfer that occurs can be summarized by the following diagram
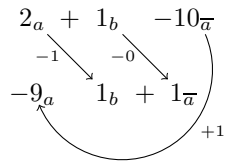
\begin{center}
\begin{tikzpicture}
\draw (0,0) node {$2_a$};
\draw (0.5,0) node {$+$};
\draw (1,0) node {$1_b$}; 
\draw (2,0) node {$-10_{\overline{a}}$};

\draw (1,-1) node {$1_b$};
\draw (1.5,-1) node {$+$};
\draw (2,-1) node {$1_{\overline{a}}$}; 
\draw (0,-1) node {$-9_{a}$};

\draw [->] (0.2,-0.2)--(0.8,-0.8);
\draw (0.2,-0.5) node {{\tiny $-1$}};
\draw [->] (1.2,-0.2)--(1.8,-0.8);
\draw (1.2,-0.5) node {{\tiny $-0$}};

\draw [->]  (2.3,-0.2) arc (23:-157:1.2);
\draw (2.4,-1.5) node {{\tiny $+1$}};

\end{tikzpicture} 
\end{center}

\end{ex}
\bi The main proposition that follows from the definition of $\Lambda$ is the following.
\begin{prop}\label{prop:switch}
 For any $(p,s)\in \Pp\times \Sc$, let us set $(s',p') = \Lambda(p,s)$.  We then have the following: 
 \begin{equation}\label{eq:crossps}
  p\not\odp s \Longleftrightarrow s'\odp p'\,,
 \end{equation}
\begin{equation}\label{eq:crosspp}
 p\not\od \gamma(s) \Longleftrightarrow \mu(s')\odp p'\,\cdot 
\end{equation}
\end{prop}
\bi
The relation \eqref{eq:crossps} means that the operator $\Lambda$ allows us to order, in terms of $\odp$, two particles of different degree which are not well-related. This property stands as the key result that will allow us to construct the mapping $\Phi$ from $\Odd$ to $\E$. On the other hand, the relation \eqref{eq:crosspp}, more subtle to explain, will play a major role in the inverse $\Psi$ of $\Phi$. 
\begin{proof}[Proof of \Prp{prop:switch}]
Let us set $p = (k,c)$ and $s= (k',c',c'')$. We then obtain $s' = (k'+\ep(c',c''),c,c')$ and $p' = (k-\ep(c,c')-\ep(c',c''),c'')$. We also observe that $\mu(s')=\gamma(s)$. We then have the following equivalences:
 \begin{align*}
  p\not\odp s &\Longleftrightarrow k-(2k'+\ep(c',c''))<\ep(c,c')+\ep(c',c'') &\text{by \eqref{ps}}\\
  &\Longleftrightarrow [2(k'+\ep(c',c''))+\ep(c,c')]-(k-\ep(c,c')-\ep(c',c''))>\ep(c,c')+\ep(c',c'')\,,\\
  &\Longleftrightarrow s'\odp p' &\text{by \eqref{sp}}\,\cdot \\\\
  p\not\od \gamma(s) &\Longleftrightarrow k-(k'+\ep(c',c''))<\ep(c,c')& \text{by \eqref{rel}}\\
  &\Longleftrightarrow k-k'\leq 1+\ep(c,c')+\ep(c',c'')\,,\\
  &\Longleftrightarrow (k'+\ep(c',c''))-(k-\ep(c,c')-\ep(c',c''))\geq 1+\ep(c',c'')\\
  &\Longleftrightarrow \mu(s')\odp p'& \text{by \eqref{pp}}\,\cdot
  \end{align*}
\end{proof}
\subsection{From $\Odd$ to $\E$}\label{sect:oe}
We now present the map $\Phi$ from $\Odd$ to $\E$. 
\m
Let us take any $\la\in \Odd$.
We set $\la = (\la_1,\ldots,\la_s)$ with $\la_{k}\od \la_{k+1}$ for any $k\in \ssss$. We illustrate this map by an example with $\C'=\{a<b\}$ and $\ep$ as described in \Expl{ex:over}:
 \[\la = (11_{\overline{b}}, 5_b,5_a,5_a,4_{\overline{a}},2_a,1_b,1_{\overline{a}},0_a,0_{\overline{b}},-1_b,-2_b)\,\cdot\]
\begin{itemize}
 \item[\Soo:]
First identify the consecutive disjoint troublesome pairs of particles ($\la_k,\la_{k+1}$ such that $\la_k\not\odp\la_{k+1}$), by beginning by those with the \textbf{smallest} potentials (from the right to the left). 

 Then, sum up these troublesome pairs $(\la_k,\la_{k+1})$ to have the secondary particles corresponding to $\la_k+\la_{k+1}$, without changing the order of the particles. We then obtain a new sequence of particles (where  particles are not necessarily well-related in terms of $\odp$) 
 $\la' = (\la'_1,\ldots,\la'_t)$, with particles $\la'_k$ in $\Odd$ and $\E$.  
In our example, we have the troublesome pairs
 \[\la = (11_{\overline{b}}, 5_b,\underbrace{5_a,5_a},4_{\overline{a}},\underbrace{2_a,1_b},\underbrace{1_{\overline{a}},0_a},\underbrace{0_{\overline{b}},-1_b},-2_b)\]
 and we obtain 
  \[\la' = (11_{\overline{b}}, 5_b,\underbrace{10_{a^2}},4_{\overline{a}},\underbrace{3_{ab}},\underbrace{1_{\overline{a}a}},\underbrace{-1_{\overline{b}b}},-2_b)\,\cdot\]
 \item[\Stt:] As long as there is a pair $(\la'_k,\la'_{k+1})\in (\Pp\times \Sc) \sqcup (\Sc\times \Pp)$ such that $\la'_k\not \odp \la'_{k+1}$, cross the particles in the pair with the operator $\Lambda$:
 \[(\la'_k,\la'_{k+1}) \longrightarrow \Lambda(\la'_k,\la'_{k+1})\,\cdot\]
 The order in which we perform the crossings is not specified here. Let us then apply this process in our example according to whether we choose the particles with the greatest or the smallest potentials for each application of $\Lambda$. We then have the following diagrams:
 
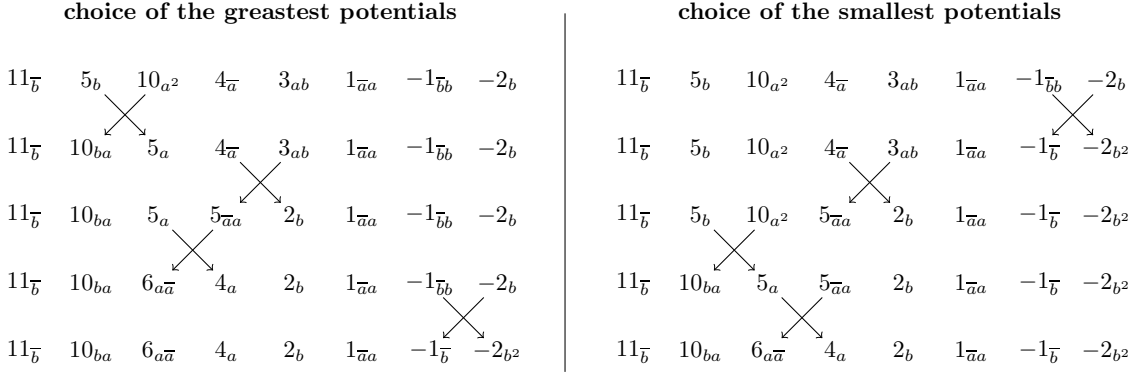
\begin{center}
 \begin{tikzpicture}[scale=0.9, every node/.style={scale=0.9}]
 \draw (0,0)--(0,-5.3);
 
 \draw (-4.5,0) node {\textbf{choice of the greastest potentials}};
 
 \draw (-8,-1) node {$11_{\overline{b}}$};
 \draw (-7,-1) node {$5_b$};
 \draw (-6,-1) node {$10_{a^2}$}; 
 \draw (-5,-1) node {$4_{\overline{a}}$};
 \draw (-4,-1) node {$3_{ab}$};
 \draw (-3,-1) node {$1_{\overline{a}a}$};
 \draw (-2,-1) node {$-1_{\overline{b}b}$};
 \draw (-1,-1) node {$-2_b$};
 \draw [->] (-6.8,-1.2)--(-6.2,-1.8);
 \draw [->] (-6.2,-1.2)--(-6.8,-1.8);
  \draw (-8,-2) node {$11_{\overline{b}}$};
 \draw (-7,-2) node {$10_{ba}$};
 \draw (-6,-2) node {$5_{a}$};
 \draw (-5,-2) node {$4_{\overline{a}}$};
 \draw (-4,-2) node {$3_{ab}$};
 \draw (-3,-2) node {$1_{\overline{a}a}$};
 \draw (-2,-2) node {$-1_{\overline{b}b}$};
 \draw (-1,-2) node {$-2_b$};
 \draw [->] (-4.8,-2.2)--(-4.2,-2.8);
 \draw [->] (-4.2,-2.2)--(-4.8,-2.8);
 \draw (-8,-3) node {$11_{\overline{b}}$};
 \draw (-7,-3) node {$10_{ba}$};
 \draw (-6,-3) node {$5_{a}$};
 \draw (-5,-3) node {$5_{\overline{a}a}$};
 \draw (-4,-3) node {$2_{b}$};
 \draw (-3,-3) node {$1_{\overline{a}a}$};
 \draw (-2,-3) node {$-1_{\overline{b}b}$};
 \draw (-1,-3) node {$-2_b$};
 \draw [->] (-5.8,-3.2)--(-5.2,-3.8);
 \draw [->] (-5.2,-3.2)--(-5.8,-3.8);
 \draw (-8,-4) node {$11_{\overline{b}}$};
 \draw (-7,-4) node {$10_{ba}$};
 \draw (-6,-4) node {$6_{a\overline{a}}$};
 \draw (-5,-4) node {$4_{a}$};
 \draw (-4,-4) node {$2_{b}$};
 \draw (-3,-4) node {$1_{\overline{a}a}$};
 \draw (-2,-4) node {$-1_{\overline{b}b}$};
 \draw (-1,-4) node {$-2_b$};
 \draw [->] (-1.8,-4.2)--(-1.2,-4.8);
 \draw [->] (-1.2,-4.2)--(-1.8,-4.8);
 \draw (-8,-5) node {$11_{\overline{b}}$};
 \draw (-7,-5) node {$10_{ba}$};
 \draw (-6,-5) node {$6_{a\overline{a}}$};
 \draw (-5,-5) node {$4_{a}$};
 \draw (-4,-5) node {$2_{b}$};
 \draw (-3,-5) node {$1_{\overline{a}a}$};
 \draw (-2,-5) node {$-1_{\overline{b}}$};
 \draw (-1,-5) node {$-2_{b^2}$};
 
 \draw (4.5,0) node {\textbf{choice of the smallest potentials}};
 
 \draw (1,-1) node {$11_{\overline{b}}$};
 \draw (2,-1) node {$5_b$};
 \draw (3,-1) node {$10_{a^2}$};
 \draw (4,-1) node {$4_{\overline{a}}$};
 \draw (5,-1) node {$3_{ab}$};
 \draw (6,-1) node {$1_{\overline{a}a}$};
 \draw (7,-1) node {$-1_{\overline{b}b}$};
 \draw (8,-1) node {$-2_b$};
 \draw [->] (7.2,-1.2)--(7.8,-1.8);
 \draw [->] (7.8,-1.2)--(7.2,-1.8);
 \draw (1,-2) node {$11_{\overline{b}}$};
 \draw (2,-2) node {$5_b$};
 \draw (3,-2) node {$10_{a^2}$};
 \draw (4,-2) node {$4_{\overline{a}}$};
 \draw (5,-2) node {$3_{ab}$};
 \draw (6,-2) node {$1_{\overline{a}a}$};
 \draw (7,-2) node {$-1_{\overline{b}}$};
 \draw (8,-2) node {$-2_{b^2}$};
 \draw [->] (4.2,-2.2)--(4.8,-2.8);
 \draw [->] (4.8,-2.2)--(4.2,-2.8);
 \draw (1,-3) node {$11_{\overline{b}}$};
 \draw (2,-3) node {$5_b$};
 \draw (3,-3) node {$10_{a^2}$};
 \draw (4,-3) node {$5_{\overline{a}a}$};
 \draw (5,-3) node {$2_{b}$};
 \draw (6,-3) node {$1_{\overline{a}a}$};
 \draw (7,-3) node {$-1_{\overline{b}}$};
 \draw (8,-3) node {$-2_{b^2}$};
 \draw [->] (2.2,-3.2)--(2.8,-3.8);
 \draw [->] (2.8,-3.2)--(2.2,-3.8);
 \draw (1,-4) node {$11_{\overline{b}}$};
 \draw (2,-4) node {$10_{ba}$};
 \draw (3,-4) node {$5_{a}$};
 \draw (4,-4) node {$5_{\overline{a}a}$};
 \draw (5,-4) node {$2_{b}$};
 \draw (6,-4) node {$1_{\overline{a}a}$};
 \draw (7,-4) node {$-1_{\overline{b}}$};
 \draw (8,-4) node {$-2_{b^2}$};
 \draw [->] (3.2,-4.2)--(3.8,-4.8);
 \draw [->] (3.8,-4.2)--(3.2,-4.8);
 \draw (1,-5) node {$11_{\overline{b}}$};
 \draw (2,-5) node {$10_{ba}$};
 \draw (3,-5) node {$6_{a\overline{a}}$};
 \draw (4,-5) node {$4_{a}$};
 \draw (5,-5) node {$2_{b}$};
 \draw (6,-5) node {$1_{\overline{a}a}$};
 \draw (7,-5) node {$-1_{\overline{b}}$};
 \draw (8,-5) node {$-2_{b^2}$};
 
 \end{tikzpicture}
 \end{center}
 One can observe in this example that the final result is the same for both choices. This is indeed the case in general, whatever the choice of the applications of $\Lambda$. 
\end{itemize}
We claim that \St always ends, and that the final result $\la''$ is unique and belongs to $\E$ (two consecutive particles are always well-related by $\odp$).
 We then set $\Phi(\la)$ to be the final partition $\la''$ obtained at the end of \Stt. In our example we have
 \[\Phi(11_{\overline{b}}, 5_b,5_a,5_a,4_{\overline{a}},2_a,1_b,1_{\overline{a}},0_a,0_{\overline{b}},-1_b,-2_b)=(11_{\overline{b}},10_{ba},6_{a\overline{a}},4_{a},2_b,1_{\overline{a}a},-1_{\overline{b}},-2_{b^2})\,\cdot\]
\subsection{From $\E$ to $\Odd$}\label{sect:eo}
Here we present the map $\Psi$ inverse of $\Phi$. Let us take any $\nu = (\nu_1,\ldots,\nu_t)\in \E$. We illustrate $\Psi$ by the example $\nu=(11_{\overline{b}},10_{ba},6_{a\overline{a}},4_{a},2_b,1_{\overline{a}a},-1_{\overline{b}},-2_{b^2})$, the final result obtained before for the map $\Phi$.
\begin{itemize}
 \item[\Soo:]
As long as there is a pair $(\nu_k,\nu_{k+1})\in \Pp\times \Sc$ such that $\nu_{k}\not \od \gamma(\nu_{k+1})$ or 
  $(\nu_k,\nu_{k+1})\in \Sc\times \Pp$ such that $\mu(\nu_k)\not \odp \nu_{k+1}$, cross the particles in the pair with $\Lambda$:
 \[(\nu_k,\nu_{k+1}) \longrightarrow \Lambda(\nu_k,\nu_{k+1})\,\cdot\]
 Here again, the order in which the applications of $\Lambda$ occur is not specified. We proceed, as before, according to whether we choose the smallest or the greatest potentials.
 
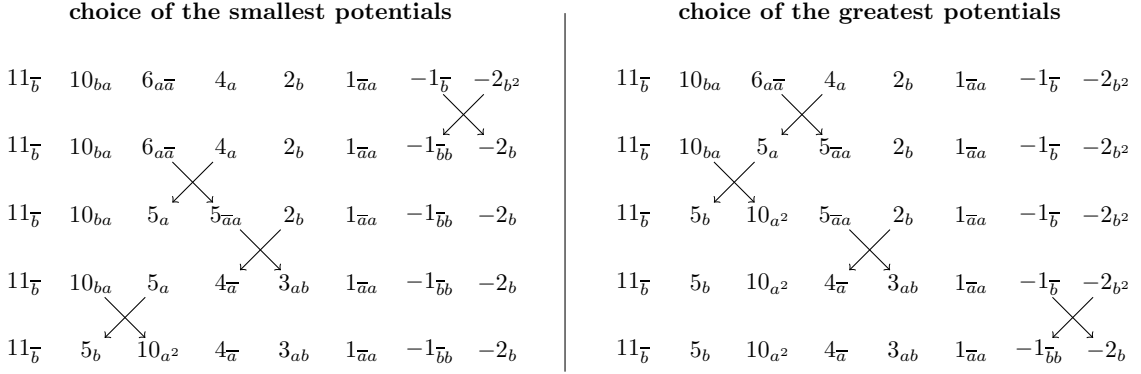
\begin{center}
 \begin{tikzpicture}[scale=0.9, every node/.style={scale=0.9}]
 \draw (0,0)--(0,-5.3);
 
 \draw (-4.5,0) node {\textbf{choice of the smallest potentials}};
 
 \draw (-8,-5) node {$11_{\overline{b}}$};
 \draw (-7,-5) node {$5_b$};
 \draw (-6,-5) node {$10_{a^2}$}; 
 \draw (-5,-5) node {$4_{\overline{a}}$};
 \draw (-4,-5) node {$3_{ab}$};
 \draw (-3,-5) node {$1_{\overline{a}a}$};
 \draw (-2,-5) node {$-1_{\overline{b}b}$};
 \draw (-1,-5) node {$-2_b$};
 \draw [->] (-6.8,-4.2)--(-6.2,-4.8);
 \draw [->] (-6.2,-4.2)--(-6.8,-4.8);
 \draw (-8,-4) node {$11_{\overline{b}}$};
 \draw (-7,-4) node {$10_{ba}$};
 \draw (-6,-4) node {$5_{a}$};
 \draw (-5,-4) node {$4_{\overline{a}}$};
 \draw (-4,-4) node {$3_{ab}$};
 \draw (-3,-4) node {$1_{\overline{a}a}$};
 \draw (-2,-4) node {$-1_{\overline{b}b}$};
 \draw (-1,-4) node {$-2_b$};
 \draw [->] (-4.8,-3.2)--(-4.2,-3.8);
 \draw [->] (-4.2,-3.2)--(-4.8,-3.8);
 \draw (-8,-3) node {$11_{\overline{b}}$};
 \draw (-7,-3) node {$10_{ba}$};
 \draw (-6,-3) node {$5_{a}$};
 \draw (-5,-3) node {$5_{\overline{a}a}$};
 \draw (-4,-3) node {$2_{b}$};
 \draw (-3,-3) node {$1_{\overline{a}a}$};
 \draw (-2,-3) node {$-1_{\overline{b}b}$};
 \draw (-1,-3) node {$-2_b$};
 \draw [->] (-5.8,-2.2)--(-5.2,-2.8);
 \draw [->] (-5.2,-2.2)--(-5.8,-2.8);
 \draw (-8,-2) node {$11_{\overline{b}}$};
 \draw (-7,-2) node {$10_{ba}$};
 \draw (-6,-2) node {$6_{a\overline{a}}$};
 \draw (-5,-2) node {$4_{a}$};
 \draw (-4,-2) node {$2_{b}$};
 \draw (-3,-2) node {$1_{\overline{a}a}$};
 \draw (-2,-2) node {$-1_{\overline{b}b}$};
 \draw (-1,-2) node {$-2_b$};
 \draw [->] (-1.8,-1.2)--(-1.2,-1.8);
 \draw [->] (-1.2,-1.2)--(-1.8,-1.8);
 \draw (-8,-1) node {$11_{\overline{b}}$};
 \draw (-7,-1) node {$10_{ba}$};
 \draw (-6,-1) node {$6_{a\overline{a}}$};
 \draw (-5,-1) node {$4_{a}$};
 \draw (-4,-1) node {$2_{b}$};
 \draw (-3,-1) node {$1_{\overline{a}a}$};
 \draw (-2,-1) node {$-1_{\overline{b}}$};
 \draw (-1,-1) node {$-2_{b^2}$};
 
 \draw (4.5,0) node {\textbf{choice of the greatest potentials}};
 
 \draw (1,-5) node {$11_{\overline{b}}$};
 \draw (2,-5) node {$5_b$};
 \draw (3,-5) node {$10_{a^2}$};
 \draw (4,-5) node {$4_{\overline{a}}$};
 \draw (5,-5) node {$3_{ab}$};
 \draw (6,-5) node {$1_{\overline{a}a}$};
 \draw (7,-5) node {$-1_{\overline{b}b}$};
 \draw (8,-5) node {$-2_b$};
 \draw [->] (7.2,-4.2)--(7.8,-4.8);
 \draw [->] (7.8,-4.2)--(7.2,-4.8);
 \draw (1,-4) node {$11_{\overline{b}}$};
 \draw (2,-4) node {$5_b$};
 \draw (3,-4) node {$10_{a^2}$};
 \draw (4,-4) node {$4_{\overline{a}}$};
 \draw (5,-4) node {$3_{ab}$};
 \draw (6,-4) node {$1_{\overline{a}a}$};
 \draw (7,-4) node {$-1_{\overline{b}}$};
 \draw (8,-4) node {$-2_{b^2}$};
 \draw [->] (4.2,-3.2)--(4.8,-3.8);
 \draw [->] (4.8,-3.2)--(4.2,-3.8);
 \draw (1,-3) node {$11_{\overline{b}}$};
 \draw (2,-3) node {$5_b$};
 \draw (3,-3) node {$10_{a^2}$};
 \draw (4,-3) node {$5_{\overline{a}a}$};
 \draw (5,-3) node {$2_{b}$};
 \draw (6,-3) node {$1_{\overline{a}a}$};
 \draw (7,-3) node {$-1_{\overline{b}}$};
 \draw (8,-3) node {$-2_{b^2}$};
 \draw [->] (2.2,-2.2)--(2.8,-2.8);
 \draw [->] (2.8,-2.2)--(2.2,-2.8);
 \draw (1,-2) node {$11_{\overline{b}}$};
 \draw (2,-2) node {$10_{ba}$};
 \draw (3,-2) node {$5_{a}$};
 \draw (4,-2) node {$5_{\overline{a}a}$};
 \draw (5,-2) node {$2_{b}$};
 \draw (6,-2) node {$1_{\overline{a}a}$};
 \draw (7,-2) node {$-1_{\overline{b}}$};
 \draw (8,-2) node {$-2_{b^2}$};
 \draw [->] (3.2,-1.2)--(3.8,-1.8);
 \draw [->] (3.8,-1.2)--(3.2,-1.8);
 \draw (1,-1) node {$11_{\overline{b}}$};
 \draw (2,-1) node {$10_{ba}$};
 \draw (3,-1) node {$6_{a\overline{a}}$};
 \draw (4,-1) node {$4_{a}$};
 \draw (5,-1) node {$2_{b}$};
 \draw (6,-1) node {$1_{\overline{a}a}$};
 \draw (7,-1) node {$-1_{\overline{b}}$};
 \draw (8,-1) node {$-2_{b^2}$};
 
 \end{tikzpicture}
 \end{center}
 We observe that the process by choosing the smallest potentials is the exact reverse process of \St of $\Phi$ by selecting the greatest potentials.  The same occurs between the choice of the greatest potentials, that gives the reverse process of \St of $\Phi$ by choosing the smallest potentials. We again have the same final result at the end of \So for both choices. 
 Let us set $\nu' = (\nu'_1,\ldots,\nu'_t)$ for the final sequence.
 \item[\textbf{Step 2}:] Split all the secondary particles $\nu'_k$ of $\nu'$ into their upper and lower halves:
 \[\nu'_k \longrightarrow \gamma(\nu'_k),\mu(\nu'_k)\,\cdot\]
 We then obtain $\nu''$. In our example, we have 
 \[\nu''= (11_{\overline{b}}, 5_b,5_a,5_a,4_{\overline{a}},2_a,1_b,1_{\overline{a}},0_a,0_{\overline{b}},-1_b,-2_b)\,\cdot\]
 \end{itemize}
 We claim that \So always ends in a unique result, whatever the choice of the applications of $\Lambda$, and that the final result $\nu''$ after \St belongs to $\Odd$  (the primary particles are well-related in terms of $\od$). We finally set $\Psi(\nu) = \nu''$ . In our example we have
 \[\Psi(11_{\overline{b}},10_{ba},6_{a\overline{a}},4_{a},2_b,1_{\overline{a}a},-1_{\overline{b}},-2_{b^2})=(11_{\overline{b}}, 5_b,5_a,5_a,4_{\overline{a}},2_a,1_b,1_{\overline{a}},0_a,0_{\overline{b}},-1_b,-2_b)\,\cdot\]
 \bi
\section{Proof of \Thm{theo:degree2}}\label{sect:proof}
In this section, we prove that the maps $\Phi$ and $\Psi$ given in \Sct{sect:maps} are well-defined and inverse to each other.
\subsection{Well-definedness of $\Phi$}\label{sect:wdphi}
Let us take any $\la = (\la_1,\ldots,\la_s) \in \Odd$, and set $\la_k = (l_k,c_k)\in \Pp$ for $k\in\sss$. Here we take the example from \Sct{sect:oe}, 
\[\la=(11_{\overline{b}}, 5_b,5_a,5_a,4_{\overline{a}},2_a,1_b,1_{\overline{a}},0_a,0_{\overline{b}},-1_b,-2_b)\,\cdot\] 
We then have $s=12$ and the following table:
\begin{equation}\label{eq:exlambda}
\begin{array}{|c|cccccccccccc|}
\hline
k&1&2&3&4&5&6&7&8&9&10&11&12\\
\hline
c_k&\overline{b}&b&a&a&\overline{a}&a&b&\overline{a}&a&\overline{b}&b&b\\
l_k&11&5&5&5&4&2&1&1&0&0&-1&-2\\
\hline
\end{array}\,\cdot
\end{equation}
In the following, we define in the first part some functions related to the partition $\la$, that will be useful for the second part, where we gives the argumentation for the proof of the well-definedness of $\Phi$.   
We explicitly compute all the functions defined in the following for our example.
\bi  
\subsubsection{The setup} We first define the function $\Delta$ on $\sss^2$ as follows,
\begin{equation}\label{del}
 \Delta : \,\,(k,k') \mapsto \left\lbrace \begin{array}{l c l}
                                     0& \text{if}& k=k'\\
                                     \displaystyle\sum_{u=k}^{k'-1} \ep(c_u, c_{u+1})& \text{if}& k<k'\\
                                     \displaystyle-\sum_{u=k'}^{k-1} \ep(c_u, c_{u+1})& \text{if}& k>k'\\
                                    \end{array}
                       \right. \,\cdot
\end{equation}
We remark that, for any $k\leq k'$,
\begin{equation}\label{signe}
 0\leq \Delta(k,k')\leq k'-k\quad, \quad\Delta(k,k')=-\Delta(k',k)\,\,,
\end{equation}
and for all $k\in \ssss$, we have by \eqref{rel} that
 \[l_k-l_{k+1}\geq \ep(c_k,c_{k+1})=\Delta(k,k+1)\,\cdot\]
Moreover, the function $\Delta$ satisfies \Crr: 
\[\Delta(k,k')+\Delta(k',k'') = \Delta(k,k'')\]
for all $k,k',k''\in\sss$.
We then identify $\Delta(k,k')$ as the \textit{formal} energy of transfer from the primary state $c_k$ to the primary state $c_k'$. Using \eqref{eq:exlambda}, we obtain the following table in our example
\begin{equation}\label{eq:exdelta}
\begin{array}{|c|ccccccccccc|}
\hline
k&1&2&3&4&5&6&7&8&9&10&11\\
\hline
\Delta(k,k+1)&1&0&0&0&1&1&0&1&0&1&0\\
\hline
\end{array}\,\cdot
\end{equation}
\bi
We now formalize the choice of troublesome pairs of primary particles in \Soo. In order to select the pairs with smallest potentials, from the right to the left, we proceed as follows:
 \begin{itemize}
  \item $i_1$ is the greatest $k\in \ssss$ such that $l_k-l_{k+1} = \Delta(k,k+1)$,
  \item if $i_{t-1}$ is selected, then, whenever it is still possible, $i_{t}$ is the greatest $k\in \{1,\ldots,i_{t-1}-2\}$ such that $l_k-l_{k+1} = \Delta(k,k+1)$.
 \end{itemize}
 We then set $I=\{i_t\}$  and  $J = \sss\setminus(I\sqcup (I+1))$. In our example, we have by \eqref{eq:exlambda} and \eqref{eq:exdelta} that
 \[i_1=10\,,\,i_2=8\,,\,i_3=6\,,\,i_4=3\,,\]
 and then 
 \[I=\{3,6,8,10\}\quad \text{and} \quad J = \{1,2,5,12\}\,\cdot\]
 \begin{rem}\label{rem:setij}
The sets $I,J$ are the unique sets $I',J'$ that satisfy the following relations:
 \begin{enumerate}
  \item $I',I'+1,J'$ form a set-partition of $\sss$,
  \item for all $i\in I'$, $l_i-l_{i+1}=\Delta(i,i+1)$,
  \item for all $j\in\Ssss\cap J'$, $l_{j-1}-l_j>\Delta(j-1,j)$.
 \end{enumerate}
 \end{rem}
By setting the function $\alpha$ on $\sss^2$ to be such that 
\begin{equation}
 \alpha : \,\,(k,k') \mapsto \left\lbrace \begin{array}{l c l}
                                     |(k,k']\cap J|& \text{if}& k\leq k'\\
                                     - \alpha(k',k)& \text{if}& k> k'
                                    \end{array}
                       \right. \,,
\end{equation}
we then have that $\alpha$ satisfies \Crr. One can also observe that $\alpha(k,k)=0$ for all $k\in\sss$. Therefore, using \Rem{rem:setij}, we obtain for all $k\leq k' \in \sss$ that 
\begin{equation}\label{ecart}
 l_{k}-l_{k'} \geq \alpha(k,k')+\Delta(k,k')\,\cdot
\end{equation}
We finally define the function $\beta$ on $\sss^2$ by
\begin{equation}\label{eq:defphi}
 \beta : \,\,(k,k') \mapsto \left\lbrace \begin{array}{l c l}
                                     |[k,k')\cap J|& \text{if}& k\leq k'\\
                                     - \beta(k',k)& \text{if}& k>k'
                                    \end{array}
                       \right. \,,
\end{equation}
and we have that $\beta$ satisfies \Crr. Our example gives the table 
\begin{equation}\label{eq:exab}
\begin{array}{|c|ccccccccccc|}
\hline
k&1&2&3&4&5&6&7&8&9&10&11\\
\hline
\alpha(k,k+1)&1&0&0&1&0&0&0&0&0&1&0\\
\hline
\beta(k,k+1)&1&1&0&0&1&0&0&0&0&0&0\\
\hline
\end{array}\,\cdot
\end{equation}
Using this table, \Cr then allows us to compute all the values for $\alpha$ and $\beta$. For example,
\[\alpha(2,4)=\alpha(2,3)+\alpha(3,4)= 0\quad \text{and}\quad \beta(4,2)=\beta(4,3)+\beta(3,2)=-0-1=-1\,\cdot\]
\bi 
To conclude, we observe that, at the end of \Soo, the particles in $\Sc$ are $\la_i+\la_{i+1}$ for $i\in I$. 
The set $I$ then corresponds to the index set of the upper halves, the set $I+1$ to the index set of the lower halves, and  $J$ represents the
 index set of the particles $\la_j$ that stay in $\Pp$. 
\subsubsection{Proof of the well-definedness of $\Phi$}
During \Stt, the positions of particles change by the actions of $\Lambda$. Here we see the secondary particles in $\Sc$ as the corresponding pair of two consecutive particles in $\Pp$. We can then consider the permutation $\sigma$ of $\sss$ which determines the new positions of these primary particles, and $\sigma $ satisfies the following properties:
\begin{itemize}
 \item $\sigma(i+1)=\sigma(i)+1$ for all $i\in I$, since we move the upper and lower halves together,
 \item $\sigma$ is increasing on $I$ and $J$, since $\Lambda$ never crosses the particles of the same degree.
\end{itemize}
\bi
We can now state the main results that will ensure the well-definedness of the map $\Phi$.
 \begin{prop}[Final positions]\label{prop:finalposphi}
 Let $\pt$ be the function on $J\times I$ defined by 
 \begin{equation}
  \pt : (j,i) \,\mapsto \, l_j-2l_{i+1}-\Delta(j,i+1)-\Delta(i+1-\beta(j,i),i+1)\,\cdot
 \end{equation}
Then the final position $\sigma$ after \St is such that for any $(j,i)\in J\times I$, 
\begin{equation}
 \sigma(j)<\sigma(i) \Longleftrightarrow \pt(j,i)\geq 0\,\cdot
\end{equation}
Futhermore, \St comes to an end after exactly
\begin{equation}\label{eq:nbcrossphi}
|\{(j,i)\in J\times I: j>i \text{ and }\pt(j,i)\geq 0\,, \text{ or } j<i \text{ and }\pt(j,i)<0\}|
\end{equation}
applications of $\Lambda$.
\end{prop}
The above proposition ensures that the process \St always ends. Using \eqref{eq:exlambda}, \eqref{eq:exdelta} and \eqref{eq:exab}, we obtain in our example the following table corresponding to $\phi$: 
\[\begin{array}{|c|cccc|}
\hline
_{j}\setminus^{i}&3&6&8&10\\
\hline
1&0&4&5&6\\
2&-5&-1&1&2\\
5&-6&-1&0&1\\
12&-8&-2&-1&0\\
\hline
\end{array}\,\cdot\]
By the proposition, we have exactly four crossings which occur in the pairs $(j,i)$ in $\{(2,3),(2,6),(5,6),(12,10)\}$, and this corresponds to the illustration of \St in \Sct{sect:oe}.
\bi
The fact that the final partition belongs to the suitable set is given by the next two propositions.
\begin{prop}\label{prop:belE}
 The partition obtained after \St belongs to $\E$.
\end{prop}
\begin{prop}\label{prop:belErho}
 For any $\rho\in\{0,1\}$, we have $\Phi(\Odd^{\rho_\pm})\subset \E^{\rho_\pm}$.
\end{prop}
\bi Before proving these propositions, we first state and show two lemmas that will be useful for the proofs of the propositions.
\begin{lem}\label{lem:1}
If a primary particle $(l_k,c_k)$ originally at position $k$ moves to position $\sigma(k)$, then it becomes \Par $(l_k+\Delta(\sigma(k),k),c_{\sigma(k)})$.
\end{lem}
\begin{lem}\label{lem:2}
 The function $\phi$ is non-increasing in the first argument in $J$ and non-decreasing in the second argument in $I$.
\end{lem}
\Lem{lem:1} plays a central role in the understanding of the operator $\Lambda$. Rephrased, it can be stated as follows: a primary particle that moves from a state  $c_k$ to a state $c_{k'}$ gains the \textbf{formal} energy of transfer from $c_k$ to $c_{k'}$. By \eqref{signe}, this energy is non-negative if $k\leq k'$, and non-positive if $k\geq k'$.
\begin{proof}[Proof of \Lem{lem:1}]
 We prove the lemma by induction on number of applications of $\Lambda$. The  energy transfer $\Ll$ conserves the state of the partition, so that the sequence of states is fixed. On the other hand, the particles gain or lose
 exactly the minimal energy needed for the transfer, and by definition, this is exactly what $\Delta$ keeps track of. As an example, if we do the transformation $\Lambda$, at position $k$, on a pair of particles in $ \Pp\times \Sc$, we obtain
\[\begin{array}{|c || c |c |c|}
\hline
   \text{initial positions}& j&i+1&i+2 \\
   \hline\hline
   \text{positions before }\Lambda& k&k+1&k+2 \\
   \hline
   \text{states before }\Lambda& c_k&c_{k+1}&c_{k+2} \\
   \hline
   \text{potentials before }\Lambda&l'_k&l'_{k+1}&l'_{k+2}\\
   \hline\hline
   \text{positions after }\Lambda& k+2&k&k+1 \\
   \hline
   \text{states after }\Lambda& c_{k+2}&c_{k}&c_{k+1} \\
   \hline
   \text{potentials after }\Lambda&\Delta(k+2,k)+l'_{k}&\Delta(k,k+1)+l'_{k+1}&\Delta(k+1,k+2)+l'_{k+2}\\
   \hline
  \end{array}\,\cdot
\]
Here we recall that $l'_{k+1}-l'_{k+2}=\Delta(k+1,k+2)$. The same calculation occurs when we consider the application of $\Lambda$ on a pair in $\Sc\times \Pp$.
 \end{proof}
\begin{proof}[Proof of \Lem{lem:2}]
We first prove that $\pt$ is non-increasing in the first argument, and then that 
$\pt$ is non-decreasing in the second argument.
\begin{itemize} 
 \item For any $j<j'\in J$ and $i\in I$, we have by \Cr and \eqref{ecart} that
 \begin{align*}
 \pt(j,i)-\pt(j',i) &= l_j-l_{j'} - \Delta(j,j')-\Delta(i+1-\beta(j,i),i+1-\beta(j',i))\\
 &\geq \alpha(j,j')-\Delta(i+1-\beta(j,i),i+1-\beta(j',i))\,\cdot
 \end{align*}
 But \Cr  and \eqref{signe} give that
 \[i+1-\beta(j',i)-(i+1-\beta(j,i)) = \beta(j,j')\geq 0\,,\]
 so that by \eqref{signe} again, we obtain that $\phi(j,i)-\phi(j',i)\geq \alpha(j,j')-\beta(j,j')$. 
 Since $j,j'\in J$, we have
 \[\alpha(j,j') = |(j,j']\cap J| = 1+|(j,j')\cap J|= |[j,j')\cap J|=\beta(j,j')\,\cdot\]
 Therefore, we always have for any $j<j'\in J$ and $i\in I$ that $\phi(j,i)-\phi(j',i)\geq 0$.\\
 \item For any $j\in J$ and $i<i'\in I$, we have by \Cr and \eqref{ecart}
 \begin{align*}
 \pt(j,i')-\pt(j,i)&= 2(l_{i+1}-l_{i'+1})-\Delta(i+1,i'+1)+\Delta(i+1-\beta(j,i),i+1)\\
 &\quad +\Delta(i'+1,i'+1-\beta(j,i'))\\
 &= 2(l_{i+1}-l_{i'+1}-\Delta(i+1,i'+1))\\
 &\quad +\Delta(i+1-\beta(j,i),i'+1-\beta(j,i'))\\
 &\geq 2\alpha(i+1,i'+1)+ \Delta(i+1-\beta(j,i),i'+1-\beta(j,i'))\,
 \end{align*}
 Since we have by \eqref{signe} that
 \begin{align*}
 i'+1-\beta(j,i')-(i+1-\beta(j,i))& = i'-i-\beta(i,i') \\
 &= |[i,i')\cap (I\sqcup (I+1))|\\&\geq 0\,\,,
 \end{align*}
 we then obtain that $\pt(j,i')-\pt(j,i)\geq 0$. 
\end{itemize}
\end{proof}
\bi
We can now prove \Prp{prop:finalposphi}, \Prp{prop:belE} and \Prp{prop:belErho}. 
\begin{proof}[Proof of \Prp{prop:finalposphi}]
Let $\sigma$ be the final position. 
\begin{itemize}
 \item Let us suppose that there exists $(j,i)\in J\times I$ such that $\sigma(j)<\sigma(i)$ and $\pt(j,i)<0$. By \Lem{lem:2} we have $\pt(j',i')<0$ for all $j<j'\in J,\,\,i'<i \in I$.
 Moreover, since $\sigma$ is increasing on $J$ and $I$, and $\sigma(J)+1\setminus \sigma(J) \subset \sigma(I)$, we necessarily have some $j<j'\in J,\,\,i'<i \in I$ such that $\sigma(j')+1 = \sigma(i')$.
 We then obtain by \Lem{lem:1} the following difference of potentials:
 \begin{align*}
 D &= \la'_{\sigma(j')}-(\la'_{\sigma(j')+1}+\la'_{\sigma(j')+2}) - \Delta(\sigma(j'),\sigma(j')+2)\\
 &=l_{j'}+\Delta(\sigma(j'),j') - [ 2 (l_{i'+1}+\Delta(\sigma(i'+1),i'+1))+\Delta(\sigma(i'),\sigma(i'+1))]\\
 &\quad-\Delta(\sigma(j'),\sigma(i'+1))\\
 &=l_{j'}-2l_{i'+1} - \Delta(j',i'+1) - \Delta(\sigma(i'),i'+1)\,\cdot
 \end{align*}
We now compute $\sigma(i')$. Since $\sigma$ is increasing on $I\sqcup(I+1)$ and on $J$, we have
\begin{align*}
\sigma(i')-1 &= \sigma(j')\\
&= |[1,j']\cap J|+ |[1,i')\cap (I\sqcup (I+1))|\\
&= 1+ \beta(j') + i'-1-\beta(i')\\
&=i'-\beta(j,i') \,\cdot
\end{align*}
Finally, we obtain by definition that $D = \pt(j',i')<0$. Since the potential difference is negative, by \eqref{ps}, we have $\la'_{\sigma(j')}\not \odp \la'_{\sigma(j')+1}+\la'_{P(j')+2}$ and $\sigma$ is no longer the final position.\\
\item
Let us now suppose that there exists $(j,i)\in J\times I$ such that $\sigma(j)>\sigma(i)$ and $\pt(j,i)\geq 0$. By \Lem{lem:2}, we have $\pt(j',i')\geq 0$ for all $j>j'\in J,\,\,i'>i \in I$.
Since $\sigma$ is increasing on $J$ and $I$, and $\sigma(J)-1\setminus \sigma(J) \subset \sigma(I)+1$, we necessarily have some $j>j'\in J,\,\,i'>i \in I$ such that $\sigma(j')-1 = \sigma(i')+1$.
We then obtain by \Lem{lem:1} the following difference of potentials:
 \begin{align*}
 D &= (\la'_{\sigma(j')-2}+\la'_{\sigma(j')-1})-\la'_{\sigma(j')}-\Delta(\sigma(j')-2,\sigma(j'))\\
 &=[ 2 (l_{i'+1}+\Delta(\sigma(i'+1),i'+1))+\Delta(\sigma(i'),\sigma(i'+1))]-l_{j'}-\Delta(\sigma(j'),j') \\
 &\quad -\Delta(\sigma(i'),\sigma(j'))\\
 &= 2l_{i'+1}-l_{j'} - \Delta(i'+1,j') - \Delta(i'+1,\sigma(i'+1))\,\cdot
 \end{align*} 
We now conpute $\sigma(i'+1)$ Since $\sigma$ is increasing on $I\sqcup(I+1)$ and on $J$,
\begin{align*}
\sigma(i'+1)+1 &= \sigma(j')\\
 &= |[1,j']\cap J|+ |[1,i'+1]\cap (I\sqcup (I+1))|\\
&= 1+|[1,j')\cap J|+ 2+|[1,i')\cap (I\sqcup (I+1))|\\
&= 2+ \beta(j') + i'-\beta(i')\\
&=2+i'-\beta(j,i') \,\cdot
\end{align*}
Finally, we obtain by definition that $D= -\pt(j',i')\leq 0$.
Since the potential difference is non-positive, by \eqref{sp}, we have $\la'_{\sigma(j')-2}+\la'_{\sigma(j')-1}\not \odp \la'_{\sigma(j')}$ and $\sigma$ is no longer the final position.
\end{itemize}
To conclude, for $\sigma$ being the last position, the first part of the reasoning gives $\sigma(j)<\sigma(i) \Longrightarrow \pt(j,i)\geq 0$ and the second part gives
$\sigma(j)<\sigma(i) \Longleftarrow \pt(j,i)\geq 0$, so that we obtain the equivalence 
\[\sigma(j)<\sigma(i) \Longleftrightarrow \pt(j,i)\geq 0\,\cdot\]
\bi One can see in the previous reasoning that for any $(j,i)\in J\times I$, whatever the choice of \Stt, 
once they meet for some position $\sigma'$ (particles have consecutive positions), we then have that the corresponding difference $D$ between the potential of the particle to the left and the potential of the particle to the right does not depend on $\sigma'$:
\begin{itemize}
\item if $\sigma'(j)+1=\sigma'(i)$, then $D = \pt(j,i)$,
\item if $\sigma'(j)-1=\sigma'(i+1)$, then $D = -\pt(j,i)$.
\end{itemize}
By \eqref{sp} and \eqref{ps}, this means that once the particles coming from $i$ and $j$ cross by $\Lambda$ in \Stt, they cannot cross back. Moreover, by the fact that the position function $\sigma'$ is increasing on $J$ and $I\sqcup (I+1)$, the crossings only occur, once, for 
$j<i$ such that $\pt(j,i)<0$ or $j>i$ such that $\pt(j,i)\leq 0$, and this gives \eqref{eq:nbcrossphi}.
\end{proof}
\begin{proof}[Proof of \Prp{prop:belE}]
 By \eqref{eq:crossps} of \Prp{prop:switch}, we obtain, by crossing two particles with different degrees  which are not well-related in terms of $\odp$, that the resulting particles become well-related in terms of $\odp$. \St then consists in ordering consecutive particles with different degrees, as the process stops as soon as this is the case. 
 \m Let us show that two consecutive primary particles are well related in terms of $\odp$. 
Since $\sigma$ is increasing on $J$, we then have, by \Crr, that for any $j<j'\in J$
 \[(l_j+\Delta(\sigma(j),j))-(l_{j'}+\Delta(\sigma(j'),j')) = l_j-l_{j'} -\Delta(j,j') + \Delta(\sigma(j),\sigma(j'))\,,\cdot\]
In particular, if $\sigma(j')=\sigma(j)+1$, we then obtain  by \eqref{ecart} and the defintion of $\alpha$ that
 \begin{align*}
 (l_j+\Delta(\sigma(j),j))-(l_{j'}+\Delta(\sigma(j'),j'))&\geq \alpha(j,j')+\Delta(\sigma(j),\sigma(j'))\\
  &= |(j,j']\cap J| + \ep(c_{\sigma(j)},c_{\sigma(j')})\\ 
  &\geq  1 + \ep(c_{\sigma(j)},c_{\sigma(j')})\,\cdot
 \end{align*}
This means, by \eqref{pp}, that two consecutive primary particles are always well-ordered in terms of $\odp$ in the final result. 
 \m Finally, with the same reasoning as before, since $\sigma$ is increasing on $I\sqcup (I+1)$, we have for $i<i'\in I$ such that $\sigma(i)+2=\sigma(i')$ that 
 \begin{align*}
 (l_{i+1}+\Delta(\sigma(i+1),i))-(l_{i'}+\Delta(\sigma(i'),i'))&\geq \alpha(i+1,i')+\Delta(\sigma(i+1),\sigma(i'))\\
  &= |(i+1,i']\cap J| + \ep(c_{\sigma(j)},c_{\sigma(j')})\\ 
  &\geq  \ep(c_{\sigma(j)},c_{\sigma(j')})\,,
 \end{align*}
 so that by \eqref{rel}, we have $\la'_{\sigma(i+1)}\od \la'_{\sigma(i')}$. We then obtain, by \eqref{ss}, that two consecutive secondary particles are always well-ordered in terms of $\odp$ in the final result.
\end{proof}
\begin{proof}[Proof of \Prp{prop:belErho}]
It suffices to show that all primary particles stay in the interval corresponding to $\rho_{\pm}$. By using \eqref{signe}, \eqref{ecart}, and \Lem{lem:1},  we obtain for any $k\in \sss$ that
 \[
  l_k + \Delta(\sigma(k),k)\leq l_1 - \alpha(1,k)-\Delta(1,\sigma(k)) \leq l_1
 \]
 and 
 \[l_k + \Delta(\sigma(k),k) \geq l_s + \alpha(k,s) + \Delta(\sigma(k),s)\geq l_s\,\cdot\]
Therefore, the potentials of the primary particles in the final partition stay in $[l_s,l_1]$. If $\la_k\in \Odd^{\rho_\pm}$ for all $k\in \sss$, then $\la'_{\sigma(k)}\in \Odd^{\rho_\pm}$ and then $\la'_{\sigma(j)}\in \Odd^{\rho_\pm}$ 
 and $\la'_{\sigma(i)}+\la'_{\sigma(i+1)}\in \E^{\rho_\pm}$ for all $(j,i)\in J\times I$.
 \end{proof}
 \bi
\subsection{Well-definedness of $\Psi$}
Let us consider $\nu\in \E$ with $\nu = (\nu_1,\ldots,\nu_t)$. We rename the indices by enumerating
all primary particles  that occur in $\nu$. This means that we count the secondary particles as a pair of consecutive primary particles. We take the example in \Sct{sect:eo}
\[\nu=(11_{\overline{b}},10_{ba},6_{a\overline{a}},4_{a},2_b,1_{\overline{a}a},-1_{\overline{b}},-2_{b^2})\,,\]
and the rewriting gives
\[\nu=(11_{\overline{b}}, \underbrace{5_b, 5_{a}},\underbrace{3_{a},3_{\overline{a}}},4_{a},2_b,\underbrace{1_{\overline{a}},0_{a}},-1_{\overline{b}},\underbrace{-1_b,-1_{b}})\,\cdot\]
As we did before for the process $\Phi$, we first give some functions related to $\nu$, and then prove the well-definedness of $\Psi$. We explicitly compute these functions for our example.
\subsubsection{The setup} 
We consider $\nu = (\nu'_1,\ldots,\nu'_s)$ written according to the primary particles that occur in $\nu$. There then exist unique sets $J,I$ such that $\sss = J\sqcup I\sqcup (I+1)$, where $J$ is the index set of the particles in $\Pp$, and $I$  and  $I+1$ are respectively the index sets 
of upper and lower halves of the particles in $\Sc$. We have In our example
\[I=\{2,4,8,11\}\quad \text{and} \quad J = \{1,6,7,10\}\,\cdot\]
We also set 
\[\nu'_k = (l_k,c_k) \text{ for all }k\in \sss\,, \]
and define the function $\Delta$ on $\sss^2$ in the same way we previously did in \eqref{del}. We finally set the function $\eta$ on $\sss^2$ to be as follows:
\begin{equation}
 \eta: \,\,(k,k') \mapsto \left\lbrace \begin{array}{l c l}
                                     |(k,k']\cap J|& \text{if}& k\leq k'\\
                                     - \eta(k',k)& \text{if}& k> k'\\
                                    \end{array}
                       \right. \,\cdot
\end{equation}
We notice that $\eta$ satisfies \Crr. In our example, we obtain the following table:
\begin{equation}\label{eq:exnu}
\begin{array}{|c|cccccccccccc|}
\hline
k&1&2&3&4&5&6&7&8&9&10&11&12\\
\hline
c_k&\overline{b}&b&a&a&\overline{a}&a&b&\overline{a}&a&\overline{b}&b&b\\
l_k&11&5&5&3&3&4&2&1&0&-1&-1&-1\\
\hline
\Delta(k,k+1)&1&0&0&0&1&1&0&1&0&1&0&\\
\eta(k,k+1)&0&0&0&0&1&1&0&0&1&0&0&\\
\hline
\end{array}\,\cdot
\end{equation}
\bi
 We now give in the following lemma the relations that link the particles' potentials.
\begin{lem}\label{lem:3} Let us set 
\[l'_k = \left\lbrace \begin{array}{l c l}
                        l_k &\text{if}& k\in J\\
                        2l_k&\text{if}&k\in I\sqcup (I+1)
                       \end{array}
              \right.\,\cdot
 \]
 Then for all $k\leq k'\in \sss$, we have 
 \begin{equation}\label{Ecart}
 l'_k-l'_{k'} \geq \eta(k,k')+\Delta(k,k')\,\cdot
 \end{equation}
In particular, for all $i\leq i' \in I\sqcup (I+1)$, we have 
\begin{equation}\label{Ecart1}
 l_i-l_{i'}\geq \Delta(i,i')\,\cdot
\end{equation}
\end{lem}
\begin{proof}
Since the functions $\eta$ and $\Delta$ satisfy \Cr, in order to show \eqref{Ecart}, it suffices to prove that for all $k\in \ssss$,
\[l'_k-l'_{k+1}\geq \beta(k,k+1)+\Delta(k,k+1)\,\cdot\]
\begin{itemize}
 \item If $k\in I$, then $k+1\in I+1$ and 
 \begin{align*}
 l'_k-l'_{k+1} &= 2\Delta(k,k+1)\\
 &\geq \Delta(k,k+1) \\
 &= \beta(k,k+1)+\Delta(k,k+1)\,\cdot
 \end{align*}
 \item If $k\in I+1$ and $k+1\in I$, then by \eqref{ss}, $(l_k,c_{k-1},c_k)\gg(l_{k+2},c_{k+1},c_{k+2})$ is equivalent to
 \begin{align*}
 l'_k-l'_{k+1} &\geq  2\Delta(k,k+1)\\
 &\geq \eta(k,k+1)+\Delta(k,k+1)\,\cdot
 \end{align*}
 \item If $k\in I+1$ and $k+1\in J$, then by \eqref{sp}, $(l_k,c_{k-1},c_k)\gg(l_{k+1},c_{k+1})$ is equivalent to
 \begin{align*} l'_k-l'_{k+1} &\geq  1+ \Delta(k,k+1)\\
 &=\eta(k,k+1)+\Delta(k,k+1)\,\cdot
 \end{align*}
 \item If $k\in J$ and $k+1\in I$, then by \eqref{ps}, $(l_k,c_k)\gg(l_{k+2},c_{k+1},c_{k+2}) $ is equivalent to
 \begin{align*}
 l'_k-l'_{k+1} &\geq  \Delta(k,k+1)\\
 &=\eta(k,k+1)+\Delta(k,k+1)\,\cdot
 \end{align*}
 \item If $k,k+1\in J$, then by \eqref{pp}, $(l_k,c_k)\gg(l_{k+1},c_{k+1})$ is equivalent to
 \begin{align*}
 l'_k-l'_{k+1} &\geq 1+\Delta(k,k+1)\\
 &=\eta(k,k+1)+\Delta(k,k+1)\,\cdot
 \end{align*}
\end{itemize}
To show \eqref{Ecart1}, we only need to prove the relation for two consecutive $i,i' \in I\sqcup I+1$. This is obvious for $i\in I$, since the following index is $i+1\in I+1$, and 
$l_i-l_{i+1}=\Delta(i,i+1)$. Now let us take $i\in I+1$. The next $i'$ (if it exists) must necessarily be in $I$, and by \eqref{Ecart}, we obtain by the definition of $\eta$ and 
\eqref{signe} that
\begin{align*}
 2(l_i-l_{i'})&= l'_i-l'_{i'} \\
 &\geq \eta(i,i')+\Delta(i,i')\\
 & = i'-i-1+\Delta(i,i')\\
 &\geq 2\Delta(i,i')-1\\
  \Longrightarrow\quad l_i-l_{i'} &\geq \Delta(i,i')- \frac{1}{2}\\
  \Longrightarrow \quad l_i-l_{i'}&\geq \Delta(i,i')\,\cdot
\end{align*}
\end{proof}
\bi
\subsubsection{Proof of the well-definedness of $\Psi$}
We can now focus on the position $\sigma$ of the particles during \So of $\Psi$.
Note that \Lem{lem:1} still holds here, as well as the fact that
$\sigma(i+1)=\sigma(i)+1$ for all $i\in I$ and $\sigma$ is increasing on $I\sqcup (I+1)$ and $J$.
\bi We now give the analogous results of \Prp{prop:finalposphi}, \Prp{prop:belE} and \Prp{prop:belOrho} that ensure the well-definedness of $\Psi$.
\begin{prop}[Final position]\label{prop:finalpospsi}
Let $\psi$ be the function on $J\times I$ defined by :
\begin{equation}
 \psi: (j,i)\longmapsto l_j-l_i - \Delta(j,i)\,\cdot
\end{equation}
Then, the final position $\sigma$ of $\Psi$ after \So is such that, for all $(j,i)\in J\times I$,
\begin{equation}
 \sigma(j)<\sigma(i) \Longleftrightarrow \psi(j,i)\geq 0\,,
\end{equation}
and \So comes to an end after exactly
\begin{equation}\label{eq:nbstp2}
|\{(j,i)\in J\times I: j>i \text{ and }\psi(j,i)\geq 0\,, \text{ or } j<i \text{ and }\psi(j,i)<0\}|
\end{equation}
applications of $\Lambda$.
\end{prop}
\begin{prop}\label{prop:belO}
The resulting partition after \St belongs to $\Odd$.
\end{prop}
\begin{prop}\label{prop:belOrho}
 For any $\rho\in\{0,1\}$, we have $\Psi(\E^{\rho_\pm})\subset \Odd^{\rho_\pm}$.
\end{prop}
\bi In our example, we obtain the corresponding table for $\Psi$:
\[\begin{array}{|c|cccc|}
\hline
_{j}\setminus^{i}&2&4&8&11\\
\hline
1&5&7&7&7\\
6&0&2&2&2\\
7&-1&1&1&1\\
10&-3&-1&-1&-1\\
\hline
\end{array}\,\cdot\]
By \Prp{prop:finalpospsi}, we have four crossings that occur in the pairs $(j,i)$ in $\{(6,2),(6,4),(7,4),(10,11)\}$.
\bi We now prove \Prp{prop:finalpospsi}, \Prp{prop:belO} and \Prp{prop:belOrho}.
\begin{proof}[Proof of \Prp{prop:finalpospsi}]
By using \Lem{lem:3}, one can easily show that $\psi$ is decreasing in the first argument in $J$ (using \eqref{Ecart}) and non-decreasing in the second argument in $I$ (using \eqref{Ecart1}). Let $\sigma$ be the final position 
\So of $\Psi$.
\begin{itemize}
 \item Let us suppose that there exists $(j,i)\in J\times I$ such that $\sigma(j)<\sigma(i)$ but $\psi(j,i)<0$.
 Since $\sigma$ is increasing on $J$ and $I$, and $\sigma(J)+1\setminus \sigma(J)\subset \sigma(I)$, there exist $(j',i')\in J\times I$ such that 
 $j<j',i'<i$ and $\sigma(j')+1 = \sigma(i')$. We also have that 
 \[\psi(j',i') \leq \psi(j',i)\leq \psi(j,i)<0\,\cdot\]
 By evaluating the potential difference at $\sigma(j')$, we obtain that
 \begin{align*}
  D&=\nu''_{\sigma(j')}-\nu''_{\sigma(j')+1} -\Delta(\sigma(j'),\sigma(j')+1) \\
  &= (l_{j'}+\Delta(\sigma(j'),j'))-(l_{i'}+\Delta(\sigma(i'),i'))-\Delta(\sigma(j'),\sigma(i'))\\
  &= l_{j'}-l_{i'}-\Delta(j',i')\\
  &=\psi(j',i')<0\,\cdot
 \end{align*}
This means by \eqref{rel} that $\nu''_{\sigma(j')}\not \od \nu''_{\sigma(j')+1}$. Since $\gamma(\nu''_{\sigma(i')}+\nu''_{\sigma(i'+1)})=\nu''_{\sigma(j')+1}$,
 we can apply $\Lambda$, so that $\sigma$ is no longer the final position. \\
 \item Let us now assume that there exists $(j,i)\in J\times I$ such that $\sigma(j)>\sigma(i)$ but $\psi(j,i)\geq 0$.
 Since $\sigma$ is increasing on $J$ and $I\sqcup(I+1)$, and $\sigma(J)-1\setminus \sigma(J)\subset \sigma(I+1)$, there exist $(j',i')\in J\times I$ such that 
 $j>j',i'>i$ and $\sigma(j')-1 = \sigma(i'+1)=\sigma(i')+1$. We also have that 
 \[\psi(j',i') \geq \psi(j',i)\geq \psi(j,i)\geq 0\,\cdot\]
 By evaluating the potential difference at $\sigma(j')$, we obtain
 \begin{align*}
  D&=\nu''_{\sigma(j')-1}-\nu''_{\sigma(j')} -\Delta(\sigma(j')-1,\sigma(j'))\\
  &= (l_{i'+1}+\Delta(\sigma(i'+1),i'+1))-(l_{j'}+\Delta(\sigma(j'),j'))-\Delta(\sigma(i'+1),\sigma(j'))\\ 
  &= l_{i'+1}-l_{j'}-\Delta(i'+1,j')\\
  &=l_{i'}-l_{j'}-\Delta(i',j')\leq 0\,\cdot
 \end{align*}
This means by \eqref{pp} that $\nu''_{\sigma(j')-1}\not \odp \nu''_{\sigma(j')}$. Since $\mu(\nu''_{\sigma(i')}+\nu''_{\sigma(i'+1)})=\nu''_{\sigma(j')-1}$,
 we can apply $\Lambda$, so that $\sigma$ is no longer the final position.
\end{itemize}
To conclude, we observe that the first part gives that $\sigma(j)<\sigma(i)\Longrightarrow \psi(j,i)\geq 0$ and the second part $\sigma(j)<\sigma(i)\Longleftarrow \psi(j,i)\geq 0$, so that we obtain the first result in \Prp{prop:finalpospsi}. 
\m We obtain \eqref{eq:nbstp2}  with the same reasoning as in the proof of \Prp{prop:finalpospsi}, by observing that the difference of potential when two particles meet does not depend on the choice in which we apply $\Lambda$, and once particles cross by $\Lambda$, they cannot cross back.   
\end{proof}
\begin{proof}[Proof of \Prp{prop:belO}]
 Since for all $k,k'\in \sss$, we obtain by \Lem{lem:1} that
 \[\nu''_{\sigma(k)}-\nu''_{\sigma(k')}-\Delta(\sigma(k),\sigma(k'))=l_k-l_{k'}-\Delta(k,k')\,\cdot\]
 Let us now consider any $k,k'$ such that $\sigma(k)+1=\sigma(k')$.
 \begin{itemize}
 \item If $(k,k')\in J^2$, we have then by \eqref{Ecart} that
 \begin{align*}
 \nu''_{\sigma(k)}- \nu''_{\sigma(k')}&\geq \eta(k,k')\\
 &= |(k,k']\cap J|\\
 &\geq 1\,,
 \end{align*}
 so that by \eqref{pp},  $\nu''_{\sigma(k)}\od\nu''_{\sigma(k')}$.
 \item If $(k,k')\in J\times I$, then since \So ended, we necessarily have 
 \[\nu''_{\sigma(k)}\od \nu''_{\sigma(k')}\,\cdot\] 
 \item If $(k,k')\in I\times I+1$, then we have 
 \[\nu''_{\sigma(k)}-\nu''_{\sigma(k')}=0\] 
 so that by \eqref{rel}, $\nu''_{\sigma(k)}\od\nu''_{\sigma(k')}$.
 \item If $(k,k')\in I+1\times J$, then since \So ended, we necessarily have 
 \[\nu''_{\sigma(k)}\odp \nu''_{\sigma(k')}\,\cdot\]
 \item If $(k,k')\in I+1\times I$, we then have by \eqref{Ecart1} that
 \[\nu''_{\sigma(k)}-\nu''_{\sigma(k')}\geq 0\]
 so that by \eqref{rel}, $\nu''_{\sigma(k)}\od\nu''_{\sigma(k')}$.
 \end{itemize}
We then obtain that $\nu'' = (\nu''_1,\ldots, \nu''_s)$ is well-ordered by $\od$ so that it belongs to $\Odd$.
\end{proof}
\begin{rem}\label{rem:revpsi}
In the latter proof, one can check that the sets $\sigma(I),\sigma(I)+1$ and $\sigma(J)$ form the unique set-partition of $\sss$
such that
\begin{enumerate}
 \item For all $i\in \sigma(I)$, $\nu''_{i}-\nu''_{i+1} = \Delta(i,i+1)$,
 \item for any $j\in \sigma(J)\cap \Ssss$, $\nu''_{j-1}\odp \nu''_{j}$.
\end{enumerate}
\end{rem}
\begin{proof}[Proof of \Prp{prop:belOrho}]
For $\rho\in\{0,1\}$, it suffices to show that $\nu''_{\sigma(k)}\geq \rho$ in the case $\rho_+$ and 
$\nu''_{\sigma(k)}\leq \rho$  in the case $\rho_-$.
\begin{itemize}
 \item If $\nu\in \E^{\rho_+}$, then, by \Lem{lem:3}, this implies that $l'_{s}\geq \rho$. For the last $j\in J$, it is easy to see by \eqref{Ecart} 
 that 
 \begin{align*}
  \nu''_{\sigma(j)} &= l'_{j}+\Delta(\sigma(j),j)\\
  &\geq l'_s + \eta(j,s)+\Delta(\sigma(j),s)\geq \rho\,\cdot
 \end{align*}
 For the last $i+1\in I+1$, we have by \eqref{Ecart} that 
 \begin{align*}
  2\nu''_{\sigma(i+1)}&= 2(l_{i+1}+\Delta(\sigma(i+1),i+1))\\
  & \geq l'_s+\eta(i+1,s)+\Delta(i+1,s) + 2\Delta(\sigma(i+1),i+1)
 \end{align*}
 but we have by definition and \eqref{signe} that $\eta(i+1,s) = s-i-1\geq \Delta(i+1,s)$, so that 
 \begin{align*}
 2\nu''_{\sigma(i+1)}&\geq l'_s + 2\Delta(\sigma(i+1),s)\\
 &\geq l'_s \\\Longrightarrow \nu''_{\sigma(i+1)}&\geq \frac{1}{2}\rho\,\cdot 
 \end{align*}
 Since $\rho\in\{0,1\}$ and $\nu''_{\sigma(i+1)}\in \mathbb{Z}$, we necessarily have that $\nu''_{\sigma(i+1)}\geq \rho$. Then for any $k\in \sss$, $\nu''_{\sigma(k)}\geq \rho$.
 \item For $\nu\in \E^{\rho_-}$, we have the following. 
 \begin{itemize}
  \item If $1\in I$, since $\sigma$ is increasing on $I\sqcup I+1$, we obtain by \eqref{Ecart1} that for all $i\in I\sqcup(I+1) $, 
  \begin{align*}
  \nu''_{\sigma(i)} &= l_i+\Delta(\sigma(i),i) \\&\leq l_1 -\Delta(1,\sigma(i))\\&\leq l_1\leq \rho\,\cdot 
  \end{align*}
  For the first $j\in J$, we have by \eqref{Ecart} that
  \begin{align*}
   \nu''_{\sigma(j)} &= l_j+\Delta(\sigma(j),j)\\
   &\leq 2l_1-\eta(1,j)-\Delta(1,\sigma(j))\\
   &\leq  2l_1-\eta(1,j)\\
   &\leq 2\rho-1\,\cdot
  \end{align*}
  Since $\rho\in \{0,1\}$, we then have that $\nu''_{\sigma(k)}\leq \rho$ for all $k\in \sss$.\\
  \item If $1\in J$, we can easily see as before that by \eqref{Ecart}, $\nu''_{\sigma(j)}\leq \rho$ for all $j\in J$.
  Now let us consider the first $i\in I$. We have by \eqref{Ecart} that 
  \begin{align*}
   2\nu''_{\sigma(i)} &= 2(l_i+\Delta(\sigma(i),i)) \\
   &\leq 2(l_i+\Delta(1,i)) \\
   &\leq l_1 - \eta(1,i)+\Delta(1,i)\\
   &=l_1-i+2+\Delta(1,i)\,\cdot
  \end{align*}
  By using \eqref{signe}, we obtain that
  \begin{align*}
  2\nu''_{\sigma(i)}&\leq \rho+1 \\
  \Longrightarrow \nu''_{\sigma(i)} &\leq \frac{ \rho+1}{2}\,,
  \end{align*}
  so that, since $\rho\in \{0,1\}$ and $\nu''_{\sigma(i)}\in \Z$, we then always have $\nu''_{\sigma(i)}\leq \rho$.
 \end{itemize}
\end{itemize}
\end{proof}
\subsection{The maps $\Phi$ and $\Psi$ are inverse of each other}
\subsubsection{The relation $\Psi\circ \Phi = Id_{\Odd}$}
For any $\la = (\la_1,\ldots,\la_s)\in \Odd$, we choose the unique sets $I,J$ such that 
\begin{enumerate}
  \item $I,I+1,J$ form a set-partition of $\sss$,
  \item for all $i\in I$, $l_i-l_{i+1}=\Delta(i,i+1)$,
  \item for all $j\in\Ssss\cap J$, $l_{j-1}-l_{j}>\Delta(j-1,j)$.
 \end{enumerate}
 Let $\sigma$ be the final position after $\Phi$. Since by \Lem{lem:1}  
 \[\la''_{\sigma(k)}-\la''_{\sigma(k')}-\Delta(\sigma(k),\sigma(k')) = l_k-l_{k'}-\Delta(k,k')\,,\]
 by considering the function $\psi$ in \Prp{prop:finalpospsi}, we obtain, for all $(j,i)\in J\times I$, that
  \begin{align*}
  j<i \Longleftrightarrow \psi(\sigma(j),\sigma(i)) &= l_j-l_{i}-\Delta(j,i)\\
  &\geq \alpha(j,i) \\
  &= |(j,i]\cap J|\\
  &\geq 0
  \end{align*}
  and
  \begin{align*}
  j>i \Longleftrightarrow \psi(\sigma(j),\sigma(i)) &= l_j-l_{i}-\Delta(j,i) \\
  &\leq -\alpha(i,j)\\
  &= -|(i,j]\cap J|\\
  &\leq -1\,,
  \end{align*}
 so that $I,J$ are exactly the final positions of $\sigma(I),\sigma(J)$ after applying $\Psi$. We then have $\Psi(\Phi(\la))=\la$.
 \bi 
 \subsubsection{The relation $\Phi\circ \Psi = Id_{\E}$}
 Let us now take any $\nu\in \E$, and let $\sigma$ be the final position after $\Psi$, and $\Psi(\nu) = \nu'' = (\nu''_1,\ldots,\nu''_s)$ with the enumeration of primary particles. We saw in \Rem{rem:revpsi} that, 
 $\sigma(I),\sigma(I)+1$ and $\sigma(J)$ form the unique set-partition of $\sss$, such that
  \begin{itemize}
   \item for all $\sigma(i)\in \sigma(I)$, $\nu''_{\sigma(i)}- \nu''_{\sigma(i)+1}= \Delta(\sigma(i),\sigma(i)+1)$,
   \item for all $\sigma(j)\in \sigma(J)\cap \Ssss$, $\nu''_{\sigma(j)-1}\odp \nu''_{\sigma(j)}$. 
  \end{itemize}
 The sets $\sigma(I)$ and $\sigma(J)$ then are exactly the unique sets we obtain after \So in the process of $\Phi$ on $\nu''$.
 Let us recall $\beta$. We have that for $k\leq k'$
 \[\beta(k,k') = |[k,k')\cap \sigma(J)|\text{ and }\beta(k,k')=-\beta(k',k)\,\cdot\]
 We then have for any $(j,i)\in J\times I$, since $\sigma$ is increasing on $J$ and $I\sqcup (I+1)$,
 \begin{align*}
 \beta(\sigma(j),\sigma(i)) &= |[1,\sigma(i))\cap \sigma(J)|-|[1,\sigma(j))\cap \sigma(J)|\\
&=\sigma(i)-1 - |[1,\sigma(i))\cap \sigma(I\sqcap (I+1))|- |[1,j)\cap J|\\
 &=\sigma(i)-1 - |[1,i)\cap (I\sqcap (I+1))|- |[1,j)\cap J|\\
 &=\sigma(i)-i + |[1,i)\cap J|- |[1,j)\cap J|\,\cdot
 \end{align*}
We then obtain by \Prp{prop:finalposphi}  and the fact that $l_{i}= l_{i+1}+\Delta(i,i+1)$
\begin{align*}
 \phi(\sigma(j),\sigma(i)) & = (l_j+\Delta(\sigma(j),j)-2(l_{i+1}+\Delta(\sigma(i+1),i+1))-\Delta(\sigma(j),\sigma(i+1))\\
 &\quad -\Delta(\sigma(i+1)-\beta(\sigma(j),\sigma(i)),\sigma(i+1))\\
 &= l_j - 2l_{i+1}-\Delta(j,i+1)-\Delta(i+1-|[1,i)\cap J|+|[1,j)\cap J|,i+1)\\
 &=l_j - 2l_{i}-\Delta(j,i)-\Delta(i+1-|[1,i)\cap J|+|[1,j)\cap J|,i)\cdot
\end{align*}
By using \eqref{signe} and \eqref{Ecart}, we obtain that
\begin{align*}
  j<i \Longleftrightarrow \phi(\sigma(j),\sigma(i)) &\geq \eta(j,i)-\Delta(i-|(j,i)\cap J|,i)\\&
  \geq |(j,i]\cap J|-|(j,i)\cap J|\\&= 0
  \end{align*}
  and 
  \begin{align*}
  j>i \Longleftrightarrow \phi(\sigma(j),\sigma(i))&\leq -\eta(i+1,j) - \Delta(i+1+|[i+1,j)\cap J|,i+1) \\
  &\leq -|(i+1,j]\cap J|+|[i+1,j)\cap J|\\
  &=-1\,\cdot
  \end{align*}
 The final positions for $\sigma(I),\sigma(J)$ after applying $\Phi$ on $\nu''$ are then exactly $I,J$. We then obtain that 
 $\Phi(\Psi(\nu))= \nu$.
\section{Closing remarks}\label{sect:remarks}
We end this paper with three remarks.
\bi
First, we consider another relation $\odg$ on $\Pp\sqcup \Sc$, which is the same as $\odp$ for \eqref{pp} and \eqref{ss}, but slightly different for other comparisons :
\begin{equation}\label{pss}
 (k,c)\odg (k'',c',c'') \Longleftrightarrow k - (2k''+\ep(c',c''))>\ep(c,c')+\ep(c',c'')
\end{equation}
\begin{equation}\label{spp}
(k',c,c')\odg (k'',c'')\Longleftrightarrow (2k'+\ep(c,c'))-k'' \geq \ep(c,c')+\ep(c',c'')\,\cdot
\end{equation}
One can easily check that, for $\ep^*(c',c) = \ep(c,c')$, we have the following:
\begin{align*}
 (k,c)\odg (k',c')&\Longleftrightarrow (-k',c')\gg_{\ep^*} (-k,c)\,,
 \\(k,c)\odg (k',c',c'') &\Longleftrightarrow (-k'-\ep^*(c'',c'),c'',c')\gg_{\ep^*}(-k,c)\,,
 \\(k,c,c')\odg (k',c'')&\Longleftrightarrow (-k',c'')\gg_{\ep^*}(-k-\ep^*(c',c),c',c)\,,
 \\(k,c,c')\odg (k',c'',c''') &\Longleftrightarrow (-k'-\ep^*(c''',c''),c''',c'')\gg_{\ep^*}(-k-\ep^*(c',c),c',c).
\end{align*}
If we define $\Ee$ to be the set of all generalized colored partitions with particles in $\Pp\sqcup\Sc$
and related by  $\odg$, we then obtain the following corollary of \textbf{Theorem 1.1}.
\begin{cor}
For any integer $n$ and any finite non-commutative product  $C$ of colors in $\C$,
 there exists a bijection between $\{\la \in \Odd: (C(\la),|\la|)=(C,n)\}$ and $\{\nu \in \Ee: (C(\nu),|\nu|)=(C,n)\}$. 
\end{cor}
While the relation $\odg$ differs from $\odp$, they both give similar difference conditions. A good example of the similarity between these relations is the fact that we can retrieve Siladi\'c's theorem by setting in the latter corollary $\C=\{a<b\}$, $\ep(i,j) = \chi(i\leq j)$ with non-negative primary part size, followed by the transformation $(q,a,b)\mapsto (q^4,q,q^3)$.
\bi
Second, we point out that another major result, the Euler distinct-odd identity, can be retrieved from \Cor{cor:2col}. Let us consider the restriction of $\C$ to the singleton $\{a\}$. The corresponding difference condition gives the matrix 
\[\bordermatrix{
\text{}&a
\cr a&0
}\]
and the corresponding generalized partitions in \Cor{cor:2col} are the classical partitions where all the parts have state $a$.
The restriction of $D'$ to the states $a,a^2$ gives the matrix
\[
\bordermatrix{
\text{}&a&a^2
\cr a&1&0
\cr a^2 &1&0
}
\,\cdot\]
One can view the corresponding partitions in $\mathcal E$ as the generalized partitions into distinct positive particles with state $a$, along with some particles with states $a^2$ having positive even potentials. In other words, we have a pair of partitions, the first partition into distinct positive  particles with state $a$, and the second into particles with positive even potential and state $a^2$. 
\m We then redo the process with the following rules.  At Step $k$, we apply the transformation $(q,a) \mapsto (q^{{2^{k-1}}},a^{2^{k-1}})$ into the identity given by the Step $1$. This leads to the following identity: the number of partitions of $n$ into particles with state $a^{2^{k-1}}$ and potential divisible by $2^{k-1}$ is equal to the number of partitions of $n$ into distinct particles with state $a^{2^{k-1}}$ and potential divisible by $2^{k-1}$, and particles with state $a^{2^k}$ and potential divisible by $2^{k}$.
\m By considering the initial Step $1$, and iterating the Steps $k$, we then have the following identity:
the number of partitions of $n$ into positive particles with state $a$ is equal to the number of partitions of $n$ into distinct particles, with the particles with states $a^{2^k} (k\in \Z_{\geq 0})$ having a potential divisible by $2^k$. We finally recover the Euler distinct-odd identity by applying in the latter identity the transformation $(q,a) \mapsto (q^2,q^{-1})$.
\bi 
Finally, we remark that the maps given in \Sct{sect:oe} and \Sct{sect:eo} differ from the variant of Bressoud's algorithm in \cite{IK19} for the generalization of Siladi\'c's theorem. In \So of $\Phi$, instead of choosing the troublesome pairs of primary particles by from the right to the left, we started in \cite{IK19} from the left to the right by first choosing the greatest potentials. This choice could have been done here. The major observation by proceeding this way is that the map $\Phi$ remains the same. This comes from the fact that the choice of troublesome pairs only depends on the maximal sub-sequences of $\la$ 
 of the form $\la_k,\ldots,\la_{k'}$,
 which satisfy $l_i-l_{i+1}=\Delta(i,i+1)$ for all $i\in \{k,\ldots,k'\}$, with the notation used in \Sct{sect:wdphi}. For such a sub-sequence with an even length, whatever the choice made, we always take the primary particles pairwise. When the length is odd, our choice  implies that we take the particles pairwise from the right to the left so that there still remains a primary particle to the left of the sequence. By crossing this primary particle with the secondary particles obtained after summing the pairs in the sequence, by \Lem{lem:1}, we exactly obtain the pairs resulting from the choice of the troublesome pairs starting from the left to the right, and the primary particle then becomes the rightmost particle of the sequence. 
\m This observation unveils a strong property that links the generalized partitions of $\Odd$ and $\E$, both kinds of partitions seen as sequences of primary particles: 
their major attribute are the maximal sequences of consecutive primary particles. In the second paper of this series, we will see how this attribute allows us to define the particles of degree $k$ for a positive $k\geq 3$, and how this definition is closely related to the notion of crystal and energy function in the Quantum theory.

 \end{document}